\documentclass[11pt,letterpaper,reqno]{amsart}
\usepackage{caption}

\usepackage[utf8]{inputenc}
\usepackage[T1]{fontenc}
\usepackage{amssymb,amsmath,mathrsfs,amsthm}
\usepackage{enumitem}
\usepackage{mathtools}
\usepackage[margin=2cm]{geometry}
\usepackage{amsfonts}
\usepackage{ dsfont }

\usepackage{esint,comment}

\usepackage[dvips]{graphicx}
\usepackage{xcolor} 
\allowdisplaybreaks
\newtheorem{thm}{Theorem}

\newtheorem{corollary}[thm]{Corollary} 
\newtheorem{lem}[thm]{Lemma}
 \newtheorem{lemma}[thm]{Lemma}

\theoremstyle{definition}
\newtheorem{defn}[thm]{Definition}
\newtheorem{rem}[thm]{Remark}

\def \no#1#2#3 {{\bf #1} (#3), #2.}
\def \eds#1#2#3 {#1, #2, #3.}

\def\R{{\mathbb R}}

\def\d{{\rm d}}

\def\N{{\mathbb N}}

\def\:{{\colon}}

\def\be#1{\begin{equation}\label{#1}}
\def\ee{\end{equation}}

\def\<{\langle}
\def\>{\rangle}
\def\coloneqq{:=}

\newcommand{\na}{\nabla}

\newcommand{\oo}{\overline{\omega}}

\newcommand{\ov}{\overline{v}}
\newcommand{\tv}{\widetilde{v}}

\newcommand{\lec}{\lesssim}
\newcommand{\gec}{\gtrsim}
\newcommand{\bs}{\begin{split}}
\newcommand{\essss}{\end{split}}

\renewcommand{\lec}{\lesssim}

\newcommand{\eqnb}{\begin{equation}}
\newcommand{\eqne}{\end{equation}}

\renewcommand{\ee}{\mathrm{e}}

\newcommand{\p}{\partial}

\renewcommand{\tt}{{\widetilde{\theta}}}

\newcommand{\tp}{{\psi}}
\newcommand{\tb}{{\phi}}
\newcommand{\otp}{{\overline{\psi }}}
\newcommand{\otb}{{\overline{\phi}}}
\newcommand{\ttp}{{\Psi }}
\newcommand{\ttb}{{\Phi}}
\newcommand{\lt}{{\widetilde{\lambda}}}
\newcommand{\nt}{{\widetilde{N}}}

\renewcommand{\R}{\mathbb{R}}

\newcommand{\Z}{\mathbb{Z}}

\renewcommand{\d}{\mathrm{d}}

\newcommand{\supp}{\operatorname{supp}}

\newcommand\blfootnote[1]{%
  \begingroup
  \renewcommand\thefootnote{}\footnote{#1}%
  \addtocounter{footnote}{-1}%
  \endgroup
}

\begin{document}
\title[Continuous loss of regularity for SQG]{Instantaneous continuous loss of regularity for the SQG equation}
\author{Diego C\'ordoba,  Luis Mart\'inez-Zoroa,  Wojciech S. O\.za\'nski} 
\maketitle
\blfootnote{\noindent D.~C\'ordoba: Instituto de Ciencias Matem\'aticas, Madrid, Spain, email: dcg@icmat.es\\
L.~Mart\'inez-Zoroa: Department of Mathematics, CUNEF university, Spain, email: luis.martinezzoroa@cunef.edu\\
W.~S.~O\.za\'nski: Department of Mathematics, Florida State University, Tallahassee, FL 32306, USA, and Department of Mathematics, Princeton University, Princeton, NJ, 08544, USA, email: wozanski@fsu.edu}

\begin{abstract}
 Given $s\in (3/2,2)$ and $\varepsilon >0$, we construct a compactly supported initial data $\theta_0$ such that $\| \theta_0 \|_{H^s}\leq \varepsilon$ and there exist $T>0$, $c>0$ and a local-in-time solution $\theta$ of the SQG equation that is compactly supported in space, continuous and differentiable   in $t$ and in $x$  on $\R^2\times [0,T]$, and, for each $t\in [0,T]$, $ \theta (\cdot ,t ) \in {H^{s/(1+ct)}}$ and $ \theta (\cdot ,t ) \not \in {H^\beta }$ for any $\beta > s/(1+ct)$. Moreover, $\theta$ is unique among all solutions with initial condition $\theta_0$ which belong to $C([0,T];H^{1+\alpha })$ for any $\alpha >0$  and is continuous and differentiable in $t$ and in $x$  on $\R^2\times [0,T]$.
\end{abstract}
\vspace{0.5cm}
{\small
Keywords: surface quasi-geostrophic equation, SQG, ill-posedness, instantaneous loss of regularity, pseudosolution, continuous loss of regularity.}
\vspace{0.5cm}
\section{Introduction}
\bibliographystyle{alpha}

We are concerned with classical solutions of  the surface quasi-geostrophic equation (SQG), \eqnb\label{sqg}
\p_t \theta + v[\theta ] \cdot \nabla \theta =0
\eqne
in $\R^2\times (0,T)$, where $T>0$ and the velocity field is given by the Biot-Savart law,
\eqnb\label{bs}
v[\theta ] (x,t) \coloneqq \mathcal{R}^T \theta (x,t) = \frac{\Gamma (3/2)}{\pi^{3/2}} \mathrm{p.v.} \int \frac{(x-y)^\perp \theta (y,t)}{|x-y|^3} \d y.
\eqne

The SQG equation arises in geophysical fluid dynamics context \cite{HPGS,pedlosky}, and, from the analysis perspective, it   shares a number of  remarkable similarities with the 3D incompressible Euler equations, as observed by  Constantin, Majda and Tabak \cite{CMT}, who also established local well-posedness in $H^s$ for $s>2$ (see also \cite{CN} for bounded domains). Moreover, Wu \cite{Wu} established local well-posedness in $C^{k,\alpha}$, for $k\geq 1$ and $\alpha \in (0,1)$. As for the critical Sobolev space $H^2$ Chae and Wu \cite{chae_wu} proved local
existence for a logarithmic inviscid regularization of SQG (see also \cite{JKM}). We note that the question of finite-time  singularity formation from smooth initial data with finite energy remains a major open problem for both the SQG equation \eqref{sqg}, as well as the 3D incompressible Euler equations.

Due to incompressibility of $v[\theta ]$ and the transport structure of \eqref{sqg} the $L^p$ norms, where $p\in [1,\infty ]$, of $\theta$ and $\| v [\theta ]\|_{L^2}$ are conserved by the evolution of \eqref{sqg} for sufficiently regular solutions.  Global existence of weak solutions in $L^2$ was proved by Resnick in \cite{resnick} (see also \cite{CN} in the case of bounded domains) and extended by
Marchand in \cite{marchand} to the class of initial data in $L^p$ with $p > 4/3$. In recent years, numerous results have been published concerning the non-uniqueness of weak solutions and the conservation of Hamiltonian systems for weak solutions for the SQG equation. For further details on these developments, we direct the reader to \cite{BSV}, \cite{BHP}, \cite{ChKL}, \cite{DP1}, \cite{DP2},  \cite{IM1}, \cite{IM2} and \cite{DGR} and the references therein.


As for ill-posedness of classical solutions of the SQG equation \eqref{sqg}, it was shown in \cite{CM} that  the equation is ill-posed in $C^k$ in the sense that given any $T>0$  one can construct solutions in $\R^2 \times [0, T)$ of \eqref{sqg} that initially are in $C^k \cap L^2$ ($k\geq 2$), but are not in $C^k$ for $t > 0$. Similar construction in $H^s$, $s\in (3/2,2]$, was also obtained in \cite{CM} (see also \cite{JK}, which shows ill-posedness in the critical space $H^2$). Additionally, instant blow-up constructions for the generalized SQG equation in Hölder spaces $C^{k,\alpha}$ 
  were obtained in \cite{CM1}, and for the SQG equation with fractional dissipation in \cite{CM2}.

We also refer the reader to the work of Elgindi and Masmoudi \cite{EM} for mild ill-posedness for perturbations of a stationary solution of SQG.   We also note results regarding norm growth in the periodic setting for SQG, such the result of Kiselev and Nazarov \cite{KN} on the existence of initial conditions with arbitrarily small norm in $H^s$ ($s \geq  11$) which give rise to a local solutions that become large after a long period of time. Friedlander and Shvydkoy \cite{FS} proved the presence of unstable eigenvalues in the spectrum. Recently, He and Kiselev \cite{HK} proved an exponential in time growth for the $C^2$ norm of the form $\sup_{t\leq T} |\nabla^2 \theta |_{L^\infty} \geq \exp (\gamma T)$, where $\gamma = \gamma (\theta_0) > 0$. Moreover, there are several rigorous constructions of non-trivial global solutions in $H^s$ ($s>2$) found in \cite{CCGS}, \cite{GS} and  \cite{ADMW}.

We note that some numerical simulations suggested the existence of solutions with very fast growth of $|\nabla \theta |$ starting with a smooth profile by a collapsing hyperbolic saddle scenario \cite{CMT}, \cite{OY} and \cite{CLSTW}). However, such a scenario cannot develop a singularity as was proved by C\'ordoba \cite{cordoba}, see also \cite{CF},
where a double exponential bound on $|\nabla \theta |$ is obtained. Another result, due to Scott \cite{scott}, discusses a blow-up scenario in which the fast growth of $|\nabla \theta |$ is associated to a cascade of filament instabilities.

In this work we are concerned with a stronger notion of ill-posedness than strong ill-posedness or nonexistence. To be precise, given $s\in (3/2,2)$ and $\varepsilon>0$ we construct initial data $\theta_0\in H^s$ with $\| \theta_0 \|_{H^s} \leq \varepsilon$ such that the unique local-in-time solution exists for $t\in [0,T]$, for some $T>0$, and loses regularity not only instantaneously at $t=0$, but continues to lose it continuously in $t$. Namely, there exists $c>0$ such that the solution $\theta (x,t)$ is such that $ \theta (\cdot ,t ) \in {H^{s/(1+ct)}}$ and $ \theta (\cdot ,t ) \not \in {H^\beta }$ for any $\beta > s/(1+ct)$.

\subsection{Main results}

We first present a norm inflation result of a smooth solution to the SQG that is localized in a small ring around the origin.

\begin{thm}[Norm inflation]\label{T00}
   Given $s\in (\frac{3}{2},2)$, there exist constants $T_0,c_0, c_{1},c_{2},c_{3},c_{4},c_{5},c_{6}>0$, such that, for any $P\in \N$, $K\geq 1$, and any sufficiently large $\lambda>1$, there exists a $P$-fold symmetric, odd-odd symmetric $\theta_0\in C_c^\infty (\R^2)$ with 
    \[
   \begin{split}
   \|\theta_0 \|_{H^{s}}&\leq c_0 P^{\frac{1}{2}}K^{-1},
   \end{split}
   \]  
 and  such that the unique solution $\theta $ to the SQG equation \eqref{sqg} with initial data $\theta_0$ exists and remains smooth until $T_0$, and
   \eqnb\label{black_box_claims}
   \begin{split}
       \|\theta\|_{H^4}&\leq C_{K,P} \lambda^{c_{5}},\\
       \|\theta\|_{L^{\infty}}&\leq \lambda^{-1},\supp\,\theta (\cdot , t) \subset B_{2\lambda^{-\frac{1}{2}}}\setminus B_{\frac{1}{2}\lambda^{-\frac{1}{2}}},\\
       \|v[ \theta]\|_{C^1},\|\theta\|_{C^1}&\leq c_{6}\log \lambda 
       \end{split}
       \eqne
       for  $t\in[0,T_0]$. Moreover, $\theta= \tb + \tp$, where $\tb,\tp$ are such that
\eqnb\label{black_box_Hb}
\| \tb \|_{H^{\beta} } \sim K^{-1} P^{\frac{1}{2}}\left( \lambda (\log \lambda)^{-1} K^{-1} \right)^{\frac{\beta - s}{s-1}}\quad \text{ and } \quad  \|\tp\|_{H^{\beta}} \sim P^{\frac{1}{2}} c_1 K^{-1-\frac{\beta t}{s-1}} \lambda^{c_{2}(\beta-s)+c_{3}\beta t } (\log \lambda)^{-c_4\beta t}
\eqne
for all $\beta\in[0,s]$, $t\in [0,T_0]$.
\end{thm}
Here and below we use function spaces (and the Biot-Savart operator \eqref{bs}) in the spatial variables, and we have omitted ``$t$'' in the notation of the left hand sides of \eqref{black_box_claims}--\eqref{black_box_Hb}. 

In our main result, we will combine an infinite sum of the initial conditions given by Theorem~\ref{T00}, concentrated around the origin, and we will use this infinite sum to show loss of regularity. We  now introduce  a notion of solution to the SQG equation \eqref{sqg} that is appropriate for the expected (low) regularity of the resulting solution.

\begin{defn}[Classical solution]\label{definition_classical_sol}
We say that a function $\theta\colon \R^2 \times [0,T] \to \R$ is a \emph{classical solution} to the SQG equation \eqref{sqg} if $\theta$ is continuous and differentiable in both $x$ and $t$ at every point $(x,t)\in \R^2\times [0,T]$, and $\theta \in L^{\infty}([0,T];H^{1+\alpha})$ for some $\alpha>0$. 
\end{defn}
Note that, with this definition, the solution needs to be differentiable but not necessarily continuously differentiable. In particular, the derivatives of the solution we construct  in our main theorem stated below  will exist for any $(x,t)$, but they will not be bounded or continuous.

\begin{thm}[Main result]\label{T01}
    Given $s\in (\frac{3}{2},2)$, $\varepsilon>0$, there exists $T>0$, $P\in \N$ and an odd-odd  symmetric, $P$-fold symmetric  classical solution $\theta$ to the SQG equation \eqref{sqg} on $\R^2\times (0,T)$ such that $\supp\, \theta(\cdot ,t)\subset B_\varepsilon$ for all $t\in [0,T]$,
    \[ \|\theta(\cdot ,0)\|_{H^{s}}\leq \varepsilon,\]
    and there exists some constant $\bar{c}>0$ such that, for $t\in[0,T]$,
    \[
    \theta (\cdot ,t) \in H^{ \frac{s}{1+ \bar{c}t}}\quad \text{ and } \quad \theta (\cdot ,t) \not \in H^{s'} \text{ for all } s'>\frac{s}{1+ \bar{c}t}.
    \]
    Furthermore, $\theta(x,t)$ is the only classical solution with the given initial conditions fulfilling 
    $$\sup_{t\in [0,T ]}\|\theta(\cdot ,t)\|_{H^{1+\delta}}<\infty$$
    for some $\delta>0$.
\end{thm}
\begin{rem}
    We use the symmetries of the solution to help us obtain several useful error estimates.  Moreover, we have uniqueness in the class of  solutions in $L^\infty_t H^{1+\delta}_x$ (for any $\delta$), even those breaking the $P$-fold symmetry.  
\end{rem}

We note that the proof of Theorem~\ref{T01} does not use any arguments by contradiction. Instead, the proof is purely constructive. The construction is sufficiently  robust to obtain very precise information about the explicit behavior of the solution. This is particularly noteworthy since Theorem~\ref{T01} is concerned with solutions of regularity below any well-posedness class. In that sense the existence of classical solutions is already non-trivial. In fact, despite low regularity class, our technique allows us to know exactly what is the regularity of solution, which makes Theorem~\ref{T01} the first result of this type in incompressible fluid mechanics. It also contrasts with other recent ill-posedness results, such the ill-posedness and nonexistence \cite{CM} for the SQG equation (in a similar regularity class of initial data), strong ill-posedness of the Euler equations in the critical spaces \cite{BL1,EJ} (see also \cite{BL2}) as well as the gap loss of Sobolev regularity for the 2D Euler equations \cite{CMO} in supercritical spaces (or see also \cite{JEuler}).  We also note a recent work \cite{JMO} on continuous loss of Sobolev regularity for the 3D incompressible Euler equations.

We emphasize that the recent results \cite{CM,CM1,CM2,CMO} regarding ill-posedness in the supercritical regime demonstrate growth due to the shear flow deformation, namely due to the component $\p_2 v_1$ or $\p_1 v_2$ of the deformation matrix $\nabla v$, which is constructed in polar coordinates $(r,\alpha)$. Namely, the ansatz for the initial data and the corresponding pseudosolution include a radial term $f(r)$ and a term of the form $g(r) \sin (N \alpha )$. This way the velocity $v[f]$ generated by $f$ has only angular component $v_\alpha [f]$, which causes shearing motion, roughly speaking, inside $\sin (N\cdot )$, which in turn causes growth of the radial derivative in time. This demonstrates the growth mechanism due to a shear flow construction.

We further emphasize that, regarding ill-posedness results in the supercritical regime in general, it is important to keep in mind not only a growth scenario of the solution, but also sufficient control of the solution. 
In the context of shear flow type constructions this is achieved, by making use of the preservation of the $N$-fold symmetry of the equation. For example, in the case of the $2$D Euler equations the radial component of the velocity can be estimated with an additional factor of $N^{-1}$, see \cite[Lemma~6]{CMO}.\\

Apart from the shear flow type growth, one also expects growth from the diagonal entries of the deformation matrix $\nabla v$. Due to the divergence-free constraint, the only possible scenario is the hyperbolic point type growth, where the velocity is of the form $(-x,y)$ near a point in $\R^2$. This is the leading mechanism allowing the result of Theorem~\ref{T00}. Such velocity causes squeezing in the $x_1$ direction and stretching in the $x_2$ direction, and it has been employed in various results in the last decade. For example, it was explored in the case of the Euler equations by \cite{BL1}, and more recently by \cite{JEuler}, as well as by \cite{BCM} in the case of the stable IPM equation.
We note that each of these examples is concerned with loss of regularity in either critical regime, or ``almost critical'' (see a related work \cite{BC}, for example). In such regularity class it is still possible to keep track of the unique solution using, for example, the Yudovich class in the case of the 2D Euler equations.\\

We emphasize that, in the context of hyperbolic point growth, it is extremely difficult to keep track of the unique solution in the strictly supercritical regime. The main reason for this is the lack of any structure of the solution which is deformed by a hyperbolic point. This is in contrast to the shear flow type growth, where $N$-fold symmetry preservation allows for sufficient bounds. \\

One of the main points of Theorem~\ref{T01} is to demonstrate that sufficient control of the solution that is being deformed by a hyperbolic point is possible in the case of the SQG equation. Since the advected scalar $\theta $ and the velocity field $v[ \theta ]$ are of the same order (i.e. $v[\theta ] = \mathcal{R}^\perp \theta$), Theorem~\ref{T01} demonstrates that one can allow the growth of the deformation matrix $\nabla v$ dictated by the hyperbolic point, while still controlling the resulting errors between the exact solution $\theta = \theta_{\rm b}+\theta_{\rm p}$ and an appropriate pseudosolution~$\overline{\theta}$ (see Section~\ref{sec_ideas} below).  Moreover the construction is robust enough to guarantee compact support in space. 

As for the uniqueness, we use novel ideas that allow us to obtain uniqueness in a class well below the usual $H^2$ regularity required for SQG.
 In particular, if the initial data constructed for Theorem~\ref{T01} gives rise to any other classical solution, then it must be, in a sense, a ``very wild'' solution, which suggests that such solutions might not even be possible. We also note that the method used to obtain uniqueness does not use very specific properties of the equation, such as the $P$-fold symmetry, and one should be able to apply similar ideas to constructions for other active scalar equations.

\subsection{Ideas of the proof} \label{sec_ideas}

The proofs of Theorems \ref{T00} and \ref{T01}, follow from three main steps:\\

\noindent\texttt{Step 1.} We construct a fundamental norm inflation scenario due to a hyperbolic point in cartesian coordinates (see Figure~\ref{fig1} for a sketch). \\

Namely, given $s\in (3/2,2)$, $K\geq 1$ and sufficiently large $\lambda >1$, we construct initial data in the form of the sum of a background initial condition $\omega (x,0)$ and perturbation initial condition $\omega_{\rm p} (x,0)$ such that there exists $T_0>0$ and constants $c_2,c_3,c_4$ such that the unique solution exists on $[0,T_0]$, and can be written as a sum of the background $\omega$ and perturbation $\omega_{\rm p}$ with
\eqnb
\begin{split}
\supp\, \left( \omega (\cdot ,t) + \omega_{\rm p} (\cdot , t) \right) &\subset B(1/\lambda ),\\
\| \omega (\cdot , t) \|_{H^\beta } &\sim K^{-1} \left( \lambda (\log \lambda )^{-1} K^{-1} \right)^{\frac{\beta -s}{s-1}},\\
 \| \omega_{\rm p} (\cdot , t) \|_{H^\beta } &\sim K^{-1-\frac{\beta t}{s-1}} \lambda^{c_2 (\beta-s)+c_3 \beta t } (\log \lambda )^{-c_4 \beta t}
\end{split}
\eqne
for all $\beta \in [0,s]$, $t\in [0,T_0]$
in which the growth visible in the last power of $\lambda$  is due to a hyperbolic velocity generated by $\omega (\cdot ,t)$. We note that the role of $K$ is to guarantee that $\| \omega (\cdot ,0)\|_{H^s} \sim \| \omega_{\rm p} (\cdot ,0)\|_{H^s} \sim K^{-1}$.\\

\noindent\texttt{Step 2.} We modify the fundamental scenario to obtain a $P$-fold symmetric norm inflation solution supported in an annulus (see Figure~\ref{fig2} for a sketch). \\

Namely, we consider the main building blocks used in Step 1, and, for a given $P\in \N$, we arrange them in a $P$-fold manner around the origin to obtain a solution $\theta$ satisfying the claim of Theorem~\ref{T00}. We emphasize that each such solution has three free parameters: the concentration parameter $\lambda>1$, $K\geq 1$, which determines the amplitude of the solution, and $P\in \N$, which gives invariance with respect to rotation by angle $2\pi/P$. In particular the solution is supported in $B(0,2\lambda^{-1/2}) \setminus  B(0,\lambda^{-1/2}/2)$, is $P$-fold symmetric, and its $H^s$ norm at $t=0$ is $P^{1/2} K^{-1}$. \\

\noindent\texttt{Step 3.} We assemble the $P$-fold symmetric norm inflation solutions of Step 2  to obtain a solution losing Sobolev regularity continuously in time. \\

Namely, we choose sufficiently large $P\in \N$ to obtain high-order cancellations near the origin, so that the solutions constructed in Step 2 (i.e. the solutions claimed by Theorem~\ref{T00}) can be assembled into an infinite sequence of annular layers. Namely we  take $K_i \coloneqq P^{1/2} 2^i / \varepsilon$ (so that the $H^s$ norm at $t=0$ of the assembled layers equals $\varepsilon$), and we choose the $\lambda_i$'s to be increasing sufficiently fast to ensure existence of the assembled solution (if $P$ is large enough) as well as uniqueness in the desired class.\\

In Sections~\ref{sec_step1}--\ref{sec_step3} below we describe Steps 1--3 (respectively) in some more detail.

\subsubsection{Step 1. Fundamental norm inflation scenario}\label{sec_step1}
 We start by considering initial conditions of the form
\eqnb\label{bg_intro}
\oo (x) \coloneqq K^{-1} \lambda^{1-s} N^{-s} g(\lambda x)\sin(\lambda N x_{1}) \sin(\lambda N x_{2}),
\eqne
 where $g\in C_c^\infty (B(0,1);[0,1]  )$ is an arbitrary nontrivial function, and $K,\lambda , N\geq 1$ are large constants. We note that the powers of $\lambda , N$ in \eqref{bg_intro} are chosen so that the $\dot{H}^{s}$ is the same for all values of $\lambda$, $N$. This is important, since we want our solution to be initially in $H^{s}$.

Furthermore, the initial condition \eqref{bg_intro} is  related to the dynamics of the SQG equation. Namely, if we take $g=1$, i.e. consider initial conditions of the form
$$ K^{-1} \lambda^{1-s} N^{-s} \sin(\lambda N x_{1}) \sin(\lambda N x_{2}),$$
then we obtain a stationary solution to the SQG equation \eqref{sqg}, since the transport term vanishes. This makes it a great starting point for norm growth generation, since we know exactly the behaviour of the solution and it generates a strong deformation at the origin.

Unfortunately, this kind of solution is not compactly supported, which is why we include the cut-off function $g(\lambda x)$ in \eqref{bg_intro}. With such a cutoff function, the solution will no longer be stationary. However, we can show that the cancellation mechanism remains strong enough to prove that
\begin{equation}\label{goodapprox}
    \omega (x,t)\approx \oo (x)
\end{equation}
in the $C^{1,\alpha}$ norm for some $\alpha >0$ and some times $t>0$, where $\omega (x,t)$ denotes the solution to the SQG equation \eqref{sqg} with initial condition \eqref{bg_ests}. Note that $C^{1,\alpha}$ is already a well-posedness class for the SQG equation (recall~\cite{Wu}), and so we expect such bound to be sufficient to keep track of the solution. The cancellation can be quantified by considering the \emph{pseudovelocity},
\eqnb\label{ov_intro}
\ov [\oo ] \coloneqq  \frac{g(\lambda x)}{\sqrt{2}K\lambda^{s-1}N^s } \begin{pmatrix}
-\sin (\lambda N x_1 )\cos (\lambda N x_2) \\
\cos (\lambda Nx_1) \sin (\lambda Nx_2)
\end{pmatrix},
\eqne
which arises when the Riesz transform $\mathcal{R}^\perp$ in the Biot-Savart law $v[\oo ] = \mathcal{R}^\perp \oo$ does not see $g$. Then 
\eqnb\label{ggg}
\begin{split}
 \ov [\oo ] \cdot \na \oo &= \frac{g(\lambda x)}{\lambda^{s-1}N^s}\sin(\lambda N x_{1}) \cos(\lambda N x_{1}) \cdot  (\p_1 g - \p_2 g) \cdot \lambda^{2-s} N^{-s} \sin(\lambda N x_{1}) \sin(\lambda N x_{1}) \\
&\lec g(\lambda x) \| g \|_{C^1} \lambda^{3-2s} N^{-2s},
\end{split}
\eqne
which gives an additional power of $N^{-1}$, as compared to the scaling of the left-hand side and the powers of $\lambda$, $N$ appearing in \eqref{bg_intro}.

It is essential that, in order to make use of the cancellation, as well as expect some growth due to the hyperbolic point, $\lambda $ and $N$ \emph{must} be related by 
\begin{equation}\label{lambdaN}
    \lambda^{2-s}N^{1-s}=K\log N.
\end{equation}
Indeed, only under this relation we do not obtain an exponential growth (in $\lambda$ or $N$) of the deformation matrix $\nabla v$. Namely, \eqref{ov_intro} shows that we should expect $\nabla v [ \oo ] = O(\lambda^{2-s } N^{1-s} ) = O(\log N)$ as $\lambda, N\to \infty$, and so, an error estimate on $\omega(t) - \oo$ gives that 
\[
\| \omega (t) - \oo \|_{C^{1,\alpha }}  \leq c \| \nabla v [\oo ] \|_{L^\infty } \| \omega (t) - \oo \|_{C^{1,\alpha }} + \text{L.O.T.} ,
\]
see \eqref{bg_c1alpha} for details, where we denoted the lower order terms by ``L.O.T.''. A Gronwall estimate then implies that
\[\| \omega (t) - \oo \|_{C^{1,\alpha }} \leq  \text{(L.O.T.)} \exp (ct \log N) = N^{ct} \text{(L.O.T.)},
\]
and so, if the lower order term involve some negative power of either $\lambda $ or $N$ (at least for small $\alpha \in (0,1)$), we obtain the $C^{1,\alpha}$ control of the error for small times, if $\alpha \in (0,1)$ is small. Moreover, we see that
\eqnb\label{def_mat_intro}
\nabla \ov [\oo ] (0) = \frac{1}{\sqrt{2}} \begin{pmatrix}
- \log N & 0 \\
0 & \log N 
\end{pmatrix},
\eqne
which shows that we should expect that the norm growth arising from the velocity $v[\oo ]$ will be, roughly speaking, of order
$$\ee^{\int_{0}^{t}(\p_{{1}}v_{1}[\oo ](0) )\d s}\approx N^{t/\sqrt{2}},$$
which is fast enough to potentially give us continuous loss of regularity, if an appropriate perturbation of \eqref{bg_intro}  is placed around the origin. This explains our choice of the relation \eqref{lambdaN}. As for the constant $K$ appearing in \eqref{lambdaN} and \eqref{bg_intro}, its role is to control the size of the $H^s$ norm of the solution, while keeping the size of the deformation matrix $\nabla v[\oo ]$ invariant. It will only become relevant in the final gluing procedure (see \eqref{choiceK} and \eqref{choiceKcons} for details).

We note that, by \eqref{lambdaN}, \eqref{ggg} becomes $\| \ov [\oo ] \cdot \na \oo \|_{L^\infty} \lec_K \lambda^{-1} N^{-2} (\log N )^2$, and similarly one can show that
\eqnb\label{cancel_C1}
\| \ov [\oo ] \cdot \na \oo \|_{C^{k,\alpha}} \lec  \lambda^{k+\alpha -1} N^{k+\alpha -2} (\log N)^2
\eqne
for each $k\geq 0$, $\alpha \in [0,1)$, see \eqref{G_ests}. In particular (see Lemma~\ref{L_holder_to_Hb} below)
\eqnb\label{cancel_Hbeta}
\| \ov [\oo ] \cdot \na \oo \|_{\dot H^\beta} \lec \| \ov [\oo ] \cdot \na \oo \|_{C^{1,\alpha }} \lambda^{\beta-\alpha -2}\lec \lambda^{\beta -2} N^{\alpha -1} (\log N)^2
\eqne
for every $\beta \in (1,2)$, $\alpha \in (\beta -1,1)$.
Similar bounds can then be obtained for $(v[\oo]-\ov [\oo ]) \cdot \na \oo $, which, combined with a careful study of the evolution of $\omega -\oo$, allow us to show \eqref{goodapprox} in appropriate function spaces, see Lemma~\ref{L_bg}.

Having established the reasons why \eqref{bg_intro} and \eqref{lambdaN} are a good choice for an almost stationary solution with a hyperbolic point of the velocity $v[\oo]$ at the origin, we now treat $\oo$ as a ``background'' solution, and consider a perturbation near the origin, which will admit growth due to the hyperbolic point.

For this, we choose the initial condition for the perturbation as 
\eqnb\label{pert_intro}
\oo_{\rm p} (x,0)\coloneqq \lt^{1-s}\nt^{-s}g(\lt x)\sin(\nt \lt x_{1}),
\eqne
where $\lt , \nt \gg \lambda$, see Figure~\ref{fig1} for a sketch. We note that, similarly to $\oo$, the $\dot H^s$ seminorm is the same for all values of $\lt, \nt$, and we can observe a very strong cancellation in their self-interaction,
$$v[\oo_{\rm p}]\cdot\nabla\oo_{\rm p}\approx 0.$$ 

However, $\oo_{\rm p}$ is anisotropic in the sense that it involves only oscillations in $x_1$, in contrast to the background $\oo$, which involves oscillations in both $x_1$ and $x_2$. The reason for this is twofold: the first reason is the fact that the  cancellation in the self-interaction resulting from initial data \eqref{pert_intro} is a little bit more robust, although this is not a crucial point of our construction. The second reason is that there is no need to keep track of the deformation of the perturbation in $x_2$, since we only expect stretching in this direction. 
In fact, the purpose of $\sin (\lt \nt x_1)$ in \eqref{pert_intro} is to capture the squeezing in the $x_1$ direction. To this end, we approximate $v[\omega ]$ by an affine vector field with deformation matrix $\na v[\omega ](0)$, namely by \[
\na v[\omega ](0) x,
\]
where we used the fact that $v[\omega ](0)=0$, and we approximate the SQG evolution from initial data \eqref{pert_intro} by passive scalar advection along velocity $v [\omega ] (0 )x$. Noting that we  can use \eqref{def_mat_intro} to approximate $\na v[\omega ](0)$, we therefore suppose that $\theta (t) = \omega (t) + \omega_{\rm p} (t)$ is a solution to the SQG equation \eqref{sqg} with initial data of the form of the sum of \eqref{bg_intro} and \eqref{pert_intro}, and we expect that 
\eqnb\label{agg}
\| \omega_{\rm p} (t) - \oo_{\rm p}(t) \|_{C^{1,a}} \ll 1
\eqne
for some times $t>0$, under an appropriate choice of sufficiently small $a\in (0,1)$ and sufficiently large $\lt, \nt$, where
\[
\oo_{\rm p} (x,t) \coloneqq \lt^{1-s}\nt^{-s}g(\lt y)\sin(\nt \lt y_{1})
\]
and
\eqnb\label{whatisy}
y = (y_1,y_2) \coloneqq 2^{-1/2} (N^{t/\sqrt{2}}x_1,N^{-t/\sqrt{2}}x_2).
\eqne
We will choose
\[
\lt = \lambda^B, \qquad \text{ and }\qquad \nt = \lt^{1-\eta },
\]
for some large constant $B\geq 1$, and small constant $\eta >0$. We choose $\eta$ sufficiently small so that the self interaction of $\omega_{\rm p}$ is under control. This is possible, since we can read from \eqref{pert_intro} that the $C^1$ norm of $\oo_{\rm p}$ is, roughly, of size $\lt^{2-s} \nt^{1-s} = \lt^{3-2s + \eta(s-1)}$, which gives a negative power of $\lt$ provided $\eta >0$ is sufficiently small, depending on $s\in (3/2,2)$. We then pick $B$ sufficiently large, so that the dynamics (and hence also error estimates) of the background $\omega (t)$ and the perturbation $\omega_{\rm p}(t)$ can be decoupled, and finally we choose $a\in (0,1)$ small enough so that, apart from the $C^1$ control we also obtain control of the left-hand side of \eqref{agg} by a negative power of $\lambda$ on some time interval, see Lemma~\ref{L_pert} for details. A direct computation then shows that 
\[
\| \omega_{\rm p} (t) - \oo_{\rm p} (t) \|_{H^s} \lec \lambda^c \lt^{-\eta }
\]
on the same time interval (see~\eqref{byinter} for details), where $c>0$ is a constant. This shows that, if $B$ is sufficiently large, the error in all Sobolev spaces $H^\beta$, $\beta \leq s$,  decays as $\lambda\to \infty$. Consequently, the regularity of the exact solution $\theta (t) = \omega (t) + \omega_{\rm p}(t)$ is dictated by the regularity of $\oo$ and $\oo_{\rm p}$ in such spaces. In particular, since any derivative of $\oo_{\rm p}$ includes a power of $N^t$ (recall~\eqref{whatisy}), we expect continuous growth in time of the derivatives of $\theta (t)$. 

\begin{center}
 \includegraphics[width=10cm]{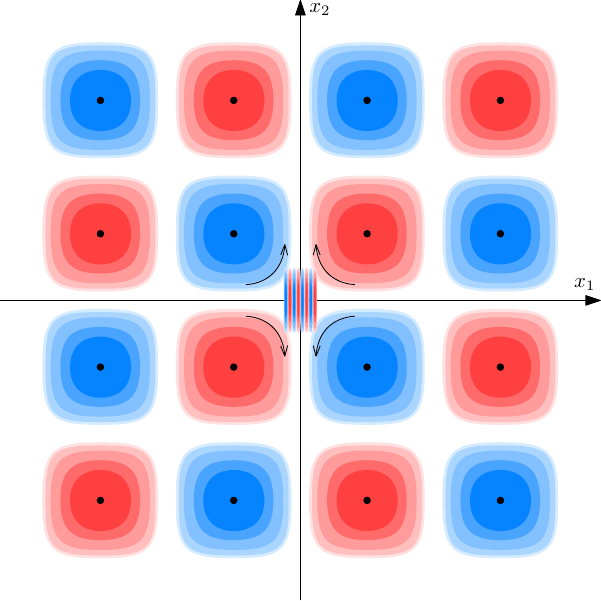}
 \end{center}
 \nopagebreak
 \captionsetup{width=.8\linewidth}
  \captionof{figure}{A sketch of the norm inflation scenario in $x_1,x_2$.}\label{fig1} 

We emphasize at this point that the above estimates can be  closed since $\theta$ is of the same order as $v[\theta ]$, which explains the point made above this section. In particular, if one considers instead a generalized SQG equation (e.g. the $2$D Euler equation), then the Biot-Savart law can introduce lower order terms that destroy the geometry considered here.\\

The above analysis suggests that we should consider a sequence of such solutions $\theta$ with an increasing sequence of $\lambda$'s and glue them to obtain continuous loss of regularity. Contrary to the previous constructions in the supercritical regime \cite{CM,CM1,CM2,CMO}, we now demonstrate that the construction is robust enough to obtain a glued solution which has compact support in space.

\subsubsection{Step 2. Norm inflation on an annulus}

 Here we discuss how the fundamental norm inflation scenario, described in the subsection above, can be modified to obtain  $P$-fold symmetry, support in space of the form of annulus of radius $\lambda^{-1/2}$, as well as quantitative estimates of \eqref{black_box_claims}. We refer the reader to Section~\ref{sec_norm_infl} for full details.

\begin{center}
 \includegraphics[width=12cm]{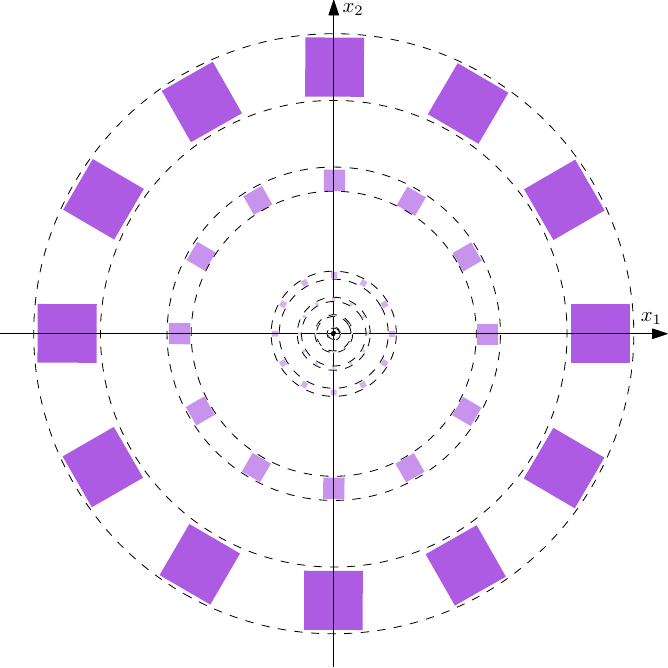}
 \end{center}
 \nopagebreak
 \captionsetup{width=.8\linewidth}
  \captionof{figure}{A sketch of the initial data giving instantaneous continuous loss of Sobolev regularity (Theorem~\ref{T01}).  The initial condition of each fundamental norm inflation solution, as depicted in Figure~\ref{fig1}, is supported in $B(0,\lambda^{-1})$, and $P$ of its translated and rotated copies are placed in an annulus of radius $\lambda^{-1/2}$, as depicted here by purple squares. The darkness of the squares corresponds to sizes of the  $L^\infty$ norms.
 A sequence of the $\lambda_i$'s is considered, resulting in infinitely many annular layers concentrating at the origin.    Here $P=12$. }\label{fig2}

Namely, we consider initial condition $\theta(\cdot ,0)$ consisting of $P$ copies of  the sum of the background \eqref{bg_intro} and the perturbation \eqref{pert_intro} (which, combined, have support in space of diameter $\sim \lambda^{-1}$), and  we arrange in a $P$-fold manner inside an annulus  of radius $\sim \lambda^{-1/2}$  around the origin, where $P\in \N$. We will refer to the sum of such copies as $\tb + \tp$, where $\tb$ corresponds to the background parts and $\tp $ to the perturbations. The main challenge of such geometry of the initial data is that it breaks the symmetry of the hyperbolic points. Namely, due to the interactions between the copies, each copy no longer enjoys the odd-odd symmetry, and each hyperbolic point is no longer stationary. In order to obtain sufficient control of such $P$-fold symmetric arrangement one needs to observe that, along particle trajectories, the deformation matrix retains its diagonal structure (see~\eqref{size_of_A}), as well as observe that the velocity field vanishes at some point near the hyperbolic point,  which is captured by Lemma~\ref{L_A} below.  This way we can control the distance traveled by particles using the $C^1$ norm of the velocity, rather than a na\"ive $L^\infty$ approach. (We note that the main idea of this trick is similar to the elementary fact that particle trajectories of a $P$-fold symmetric solution to the SQG equation that is supported away from $0$, $P\geq 2$, do not approach the origin in finite time, see Lemma~\ref{L_approach} for details.) This way we can ensure that the particle trajectories do not drift far away from the initial positions for some $t>0$ (see \eqref{eta_gronwall}), which is a sufficiently strong estimate to compensate for the symmetry breaking. Consequently, we obtain the claim of the norm inflation result, Theorem~\ref{T00}.\\

\subsubsection{Step 3. Gluing infinitely many annular layers}\label{sec_step3}

 Let us now describe how to assemble a sequence of $P$-fold symmetric norm inflation solutions $\theta_\lambda$ given by Theorem~\ref{T00} to obtain instantaneous continuous loss of Sobolev regularity (Theorem~\ref{T01}). 

Each such solution $\theta_\lambda$ is supported on an annulus of radius $\lambda^{-1/2}$ (see Figure~\ref{fig2}, where each ring corresponds to a rapidly growing solution, and each purple square corresponds to the piece of a solution evolving from an initial condition defined in \eqref{bg_intro}, \eqref{pert_intro}. The gluing requires  two essential lemmas to obtain existence and uniqueness of solution propagating from the glued initial data.\\

The existence lemma (Lemma~\ref{existence}) states, roughly speaking, that, if $\theta $ is a $P$-fold symmetric solution to the SQG equation \eqref{sqg} that is supported in $B_R\setminus B_r$ for some $0<r<R$, and $\theta_\lambda$ is given by Theorem~\ref{T00}, then for sufficiently large $\lambda$ the solution $\theta_{\rm new}$ of the SQG equation with initial data $\theta (0) + \theta_\lambda (0)$ exists on the same time interval and differs from $\theta(t) + \theta_\lambda (t)$ in a high regularity norm (such as $H^3$, for example) by an error of order $\lambda^{-1}$. In such a lemma it is essential to use the fact that 
$$|v[\theta] (x,t) |\leq O(|x|^{P-1})$$
as $x\to 0$, and consider sufficiently large $P$, so that this decay allows the two parts $\theta (0)$, $\theta_\lambda(0)$ of the initial data to evolve independently up to a small error  (to be precise $P$ must be greater or equal $C+12$, where $C$ is the universal constant appearing in \eqref{hereisP} below, so that we can obtain the estimate \eqref{theta_006}). The lemma enables us to consider the solution propagating from the first $n$ layers of the initial condition, and add the $(n+1)$-st layer. As a consequence, we can prove that the interactions between the layers remain small enough to show convergence in the limit $n\to \infty$, see Fig.~\ref{fig2}. Moreover, it ensures the  convergence is good enough to show that the limit function is a classical solution in the sense of Definition~\ref{definition_classical_sol}.\\

The uniqueness lemma (Lemma~\ref{uniqueness}) states that, if $\theta $ is a solution to the SQG equation \eqref{sqg} supported in $B_R\setminus B_r$ for some $0<r<R$ and $\theta_\lambda (x)$ is an initial condition supported in $B_{2\lambda^{-1/2}}$ such that $\| \theta_\lambda \|_{\infty }\leq 2\lambda^{-1}$, then  any solution $\theta_{\rm new}\in L^\infty_t H^{1+\alpha }_x $ (where $\alpha>0$ is arbitrary) of the SQG equation with initial data $\theta (0)+\theta_\lambda$,  remains  arbitrarily close to $\theta $ on $B^c_{r/2}$, if $\lambda$ is sufficiently large. In this lemma, it is essential that $\theta_\lambda$ is treated only as initial condition with given support  and with given $L^\infty$ norm estimate. Indeed, here  $\theta_\lambda$ represents all layers, from, say, $n+1$ to $\infty$, and so its evolution in time will involve self-interactions between all layers, and as such the particle trajectories cannot be estimated. Instead, we suppose that $\theta_{\rm new}(t) = \widetilde{\theta}_\lambda (t) + \widetilde{\theta} (t)$, where  $\widetilde{\theta}_\lambda (0)= \theta_\lambda$ and $\widetilde{\theta } (0) = \theta(0)$, and we observe that the particles originating in the support of $\widetilde{\theta}_\lambda$ can be estimated using the inequality
\[
\| v [\widetilde{\theta}_\lambda ] \|_{\infty } \lec_\alpha \|\widetilde{\theta}_\lambda\|_\infty \log (2+ \| \widetilde{\theta}_\lambda \|_{H^{1+\alpha }} )
\]
(see \eqref{repl}). This trick allows us to use the $L^\infty$ conservation $\|\widetilde{\theta}_\lambda (t) \|_\infty \leq \| \theta_\lambda \|_\infty$ and the assumption on the regularity class $L^\infty_t H^{1+\alpha }_x$ of $\theta_{\rm new}$ to see that  $\widetilde{\theta}_\lambda$ will not increase its support to anywhere near $B_{r/2}^c$. Thus $\widetilde{\theta}_\lambda$ will be only exerting very small influence on $\theta_{\rm new}$ outside of $B_{r/2}$, and so, for sufficiently large $\lambda$, $\theta_{\rm new}$ will remain close to $\theta$ on $B^c_{r/2}$, as required. 
In other words, for any solution in $H^{1+\alpha}$ the inner layers will move slowly enough to not change their support too much, and thus their interactions with the outer layers have to remain small for the times considered.

In particular, the uniqueness lemma implies that nonuniqueness cannot arise from any of the layers, so that it could only propagate from the origin. However, this can also be excluded due to the assumed regularity of solutions (in Definition~\ref{definition_classical_sol}), see the comments following \eqref{uniqueness_appl}.  We also refer the reader to the discussion following the proof of Lemma~\ref{existence} for an explanation of how the existence lemma inspires the uniqueness lemma.\\

Using the existence and uniqueness lemmas allows us to construct a sequence of $\lambda_{n}$'s for which the initial condition $\sum_{n\geq 1} \theta_{\lambda_n} (0)$ gives rise to a unique classical solution to the SQG equation \eqref{sqg}, where $\theta_{\lambda_n}$ are the norm inflation solutions given by Theorem~\ref{T00}, if $P$ is large enough. The claim of continuous loss of regularity in Theorem~\ref{T01} then follows directly from the quantitative estimates \eqref{black_box_Hb}, see~\eqref{ooo} for details.\\

The structure of the paper is as follows. We first introduce some preliminary concepts in Section~\ref{sec_prelims}. We then discuss some properties of functions with oscillations in Section~\ref{sec_osc_fcns}, and we prove Theorem~\ref{T00} in Section~\ref{sec_norm_infl} and Theorem~\ref{T01} in Section~\ref{sec_pf_main}.

\section{Preliminaries}\label{sec_prelims}

We will denote the tensor of all partial derivatives (in space) of order $J\geq 0$ by ``$D^J$''. We will denote by $c>0$ a generic absolute constant, which can change value from line to line. We will use ``$\supp$'' to refer to the support only in spatial variables, i.e.
\[ \supp\, f(x,t)=\{x\in \R^2 \colon  x\in \supp\, f(\cdot ,t)\}.\]

Given a function $f\colon \R^2 \to \R$ we will say it is \emph{$P$-fold symmetric} if it is $2\pi/P$-periodic in the angular variable $\alpha$, i.e., using polar coordinates,
$$f(r,\alpha)=f\left( r,\alpha+\frac{2\pi}{P} \right) $$
for every $r>0$, $\alpha \in [0,2\pi )$. We will often make use of the following Gronwall type estimate:
\eqnb\label{ode_fact}
\text{if }f'(t) \leq c f(t) + b \quad \text{ and }\quad f(0)=0, \qquad \text{ then }\quad f(t) \leq \frac{b}{c} \left( \ee^{ct}-1\right) \leq bt \ee^{ct}.
\eqne
We will use several basic properties of the velocity operator $v$ given by \eqref{bs}, namely, for $\alpha\in(0,1)$, $k\in\N$, $s\geq 0$
$$\| v[\theta]\|_{C^{k,\alpha}}\lec_{k,\alpha}\|\theta\|_{C^{k,\alpha}},\qquad \|v[\theta]\|_{H^s}\lec_s\|\theta\|_{H^{s}},$$
\eqnb\label{vel_est}
\| v[g] \|_{L^\infty } \lec_\alpha \| g \|_{L^\infty} \log (2+\| g\|_{C^{1,\alpha }} ) \quad \text{ and }\quad \| v[g] \|_{C^1 } \lec_\alpha \| g \|_{C^1 } \log (2+\| g\|_{C^{1,\alpha }} ).
\eqne

Finally, we recall the Sobolev-Slobodeckij characterization
\eqnb\label{ss}
\| f \|_{\dot H^s }^2 = C_s \int_{\R^2} \int_{\R^2} \frac{|f(x)-f(y)|^2}{|x-y|^{2+2s}}   \d x \,\d y \qquad \text{ for }s\in (0,1),
\eqne
see \cite[Proposition~3.4]{hitchhiker} for a proof.  We can use this to prove the following.

\begin{lemma}[$\dot H^s$ norm of a sum]\label{L_ss_cons}
Let $\{ f_j \}$ be such that $\supp\,f_j \subset B_{2R_j}\setminus B_{R_j}$ for some sequence $\{ R_j \} \subset (0,\infty )$ with $R_{j+1}\leq R_{j}/4$ for all $j\geq 1$. Then
\[
\left| \left\| \sum_{j\geq 1} f_j \right\|_{\dot H^s}^2 - \sum_{j\geq 1} \left\|  f_j \right\|_{\dot H^s}^2 \right| \lec_{s,\delta ,R_1}  \sum_{j\geq 1} R_j^{-2s-\delta } \| f_j \|_{L^2}^2
\]
for every $s\in (0,1)$ and every $\delta \in (0,1)$.
\end{lemma}
\begin{proof}
We set
\[
K \coloneqq \sum_{j\geq 1} \int_{B_{4R_j}\setminus B_{R_j/2}} \int_{B_{4R_j}\setminus B_{R_j/2}} \frac{|f_j(x)-f_j(y)|^2}{|x-y|^{2+2s } }\d x \, \d y ,
\]
and we observe that
\[
\begin{split}
&\left| \sum_{j\geq 1} \left\|  f_j \right\|_{\dot H^s}^2 - K\right| =\left| \sum_{j\geq 1} \int_{\R^2} \int_{\R^2} \frac{|f_j(x)-f_j(y)|^2}{|x-y|^{2+2s } }\d x \, \d y - K\right| \\
&\hspace{1cm}\leq 2 \sum_{j\geq 1}  \int_{B_{R_j/2}} \int_{B_{4R_j}\setminus B_{R_j/2}} \frac{|f_j(x)|^2}{|x-y|^{2+2s } }\d x \, \d y+2\sum_{j\geq 1}  \int_{B_{4R_j}^c} \int_{B_{4R_j}\setminus B_{R_j/2}} \frac{|f_j(x)|^2}{|x-y|^{2+2s } }\d x \, \d y\\
&\hspace{1cm}\lec \sum_{j\geq 1} \left( \frac{1}{R_j^{2+2s}}  \int_{B_{R_j/2}} \int_{B_{4R_j}\setminus B_{R_j/2}} |f_j(x)|^2\d x \, \d y+ R_j^{-2s} \int_{B_{4R_j}\setminus B_{R_j/2}} {|f_j(x)|^2}\d x \right)\\
&\hspace{1cm}\lec \sum_{j\geq 1} R_j^{-2s} \| f_j \|_{L^2}^2 .
\end{split}
\]
As for $\left\| \sum_{j\geq 1} f_j \right\|_{\dot H^s}^2 - K$ we have 
\[
\R^2 \times \R^2 \setminus A_j \times A_j = \left( \left( \bigcup_j A_j \right)^c \times \R^2 \right)  \cup \bigcup_{j\geq 1} \left( A_j \times A_j^c \right) ,
\]
where $A_j \coloneqq B_{4R_j } \setminus R_{R_j/2}$. As for the first set on the right-hand side we have
\[
\begin{split}
&\int_{\left( \bigcup_j A_j\right)^c } \int_{\R^2 } \frac{|\sum_{k\geq 1}f_k (x) -\sum_{m\geq 1 } f_m(y)|^2}{|x-y|^{2+2s } }\d x \, \d y \\
&= \int_{B_{4R_1}^c } \int_{\R^2 } \frac{|\sum_{k\geq 1}f_k (x)|^2}{|x-y|^{2+2s } }\d x \, \d y +\sum_{m\geq 1}\int_{B_{R_{m}/2} \setminus B_{4R_{m+1}} } \int_{\R^2 } \frac{|\sum_{k\geq 1}f_k (x) |^2}{|x-y|^{2+2s } }\d x \, \d y \\
&\lec \left(  \int_{B_{4R_1}^c } \frac{1}{|y|^{2+2s } } \d y \right) \sum_{k\geq 1} \| f_k \|_{L^2}^2 \\
&+\sum_{\substack{m,k\geq 1 \\ m\geq k }}\int_{B_{R_{m}/2} \setminus B_{4R_{m+1}} } \int_{\R^2 } \frac{|f_k (x) |^2}{R_k^{2+2s } }\d x \, \d y+\sum_{\substack{m,k\geq 1 \\ m< k }}\int_{B_{R_{m}/2} \setminus B_{4R_{m+1}} } \int_{\R^2 } \frac{|f_k (x) |^2}{|y|^{2+2s } }\d x \, \d y \\
&\lec R_1^{-2s} \sum_{k\geq 1} \| f_k \|_{L^2}^2 +\sum_{k\geq 1 }R_k^{-2-2s} \| f_k \|_{L^2}^2 \int_{B_{R_k}}\d y +\sum_{k\geq 1 } \| f_k \|^2 \int_{B_{4R_k}^c} |y|^{-2-2s} \d y   \\
&\lec \sum_{k\geq j } R_k^{-2s} \| f_k \|_{L^2}^2,
\end{split}
\]
and, for the second set, we have
\[
\begin{split}
&\int_{A_j} \int_{A_j^c}\frac{|\sum_{k\geq 1}f_k (x) -f_j(y)|^2}{|x-y|^{2+2s } }\d x \, \d y \\
&= \int_{B_{4R_j}\setminus B_{R_j/2}} \int_{B_{R_j/2}}\frac{|\sum_{k>j}f_k (x) -f_j(y)|^2}{|x-y|^{2+2s } }\d x \, \d y + \int_{B_{4R_j}\setminus B_{R_j/2}} \int_{B_{4R_j}^c}  \frac{|\sum_{k=1}^{j-1}f_k(x) -f_j(y)|^2}{|x-y|^{2+2s } }\d x \, \d y \\
&\lec \int_{B_{4R_j}\setminus B_{R_j/2}} \int_{B_{R_j/2}}\left( \frac{|\sum_{k>j}f_k (x) |^2 +|f_j(y)|^2}{R_j^{2+2s } }\right) \d x \, \d y + \int_{B_{4R_j}\setminus B_{R_j/2}} \int_{B_{4R_j}^c}  \frac{|\sum_{k=1}^{j-1}f_k(x) |^2+|f_j(y)|^2}{|x|^{2+2s } }\d x \, \d y \\
&\sim R_j^{-2s} \sum_{k\geq j }  \| f_k \|_{L^2}^2 + R_j^{-2s } \int_{B_{4R_j}\setminus B_{R_j/2}} |f_j(y)|^2 \d y  +  \sum_{k=1}^{j-1} \int_{B_{2R_k}\setminus B_{R_k}} \frac{|f_k(x)|^2}{|x|^{2+2s}} \d x \\
&\lec \sum_{k\geq j } R_k^{-2s} \| f_k \|_{L^2}^2.
\end{split}
\]
Summing in $j$ this gives
\[\begin{split}
\iint_{\bigcup \left( A_j \times A_j^c \right) } \frac{|\sum_{k\geq 1}f_k (x) -\sum_{m\geq 1}f_m (x)|^2}{|x-y|^{2+2s } }\d x \, \d y &\lec \sum_{j\geq 1} \sum_{k\geq j } R_k^{-2s} \| f_k \|_{L^2}^2 \\
&\hspace{-2cm}= \sum_{k\geq 1} k  R_k^{-2s} \| f_k \|_{L^2}^2 \lec_{\delta ,R_1}  \sum_{k\geq 1}   R_k^{-2s-\delta } \| f_k \|_{L^2}^2.
\end{split}
\]
Summing the two inequalities above, we thus obtain
\[
\left| \left\| \sum_{j\geq 1} f_j \right\|_{\dot H^s}^2 - K\right| \lec_{s,\delta ,R_1} \sum_{k\geq j } R_k^{-2s-\delta} \| f_k \|_{L^2}^2,
\]
as required.
\end{proof}

We note that, given an affine velocity field,
\[
u=A(t) x +b(t),
\]
the solution to $f_t + u\cdot \nabla f =0$ is
\eqnb\label{transport1}
f(x,t) = f_0 \left( \ee^{-\int_0^t A(s) \d s} x - \int_0^t \ee^{-\int_0^s A(\tau ) \d \tau } b(s) \d s \right).
\eqne
Conversely, $f(x,t)\coloneqq f_0 (\ee^{-A(t)} x -b(t))$ satisfies $f_t + u\cdot \nabla f=0$ with
\eqnb\label{transport2}
u = A'(t) x + \ee^{A(t)}b'(t).
\eqne

\begin{lemma}[H\"older to $H^\beta$]\label{L_holder_to_Hb}
If $G\in C^{1,\alpha } (\R^2)$ and $\supp\,G \subset B_R$ then
\[
\| G \|_{\dot H^\beta} \lec \| G \|_{C^{1,\alpha}} R^{2-\beta +\alpha }
\]
for every $\beta \in (1,1+\alpha)$.
\end{lemma}
\begin{proof} Using the Sobolev-Slobodeckij characterization \eqref{ss}, we obtain  
\[
\begin{split}
\| G \|_{\dot H^\beta}^2 &= C_{\beta } \int_{\R^2} \int_{\R^2} \frac{|\nabla G(x)-\nabla G(y)|^2}{|x-y|^{2\beta}} \d x \,\d y\\
&\leq \| G \|_{C^{1,\alpha}}^2 \int_{B_{2R}} \int_{B_{2R}}  \frac{\d x \, \d y}{|x-y|^{2\beta-2\alpha}}  + 2\int_{B_{2R}} \int_{B_{2R}^c } \frac{| \nabla G(y) - \nabla G(0,2R)|^2\d x \,\d y}{|x-y|^{2\beta}}\\
&\lec \| G \|_{C^{1,\alpha}}^2 R^{4-2\beta +2\alpha}  , 
\end{split}
\]
as required.
\end{proof}

\begin{lemma}[Trajectories approaching the origin]\label{L_approach}
Suppose that $\theta $ is a $P$-fold symmetric solution of the SQG equation with $P\geq 2$ and let $\phi$ denote the particle trajectory of $v[\theta ]$. Then
\[
|\phi (x,t) | \geq \ee^{-\int_0^t \| \nabla v[\theta ] \|_\infty \d s} |x|.
\]
In particular, if $\supp\,\theta(\cdot ,0) \subset B_r^c$ then $\supp\,\theta(\cdot ,t) \subset B_{r\exp (-\int_0^t \| \nabla v[\theta ]\|_\infty \d s )}^c$.
\end{lemma}
\begin{proof}
 We have that 
\[
\frac{\d }{\d t} |\phi | \geq - |v[\theta ] (\phi )| \geq - \| \nabla v[\theta ] \|_{\infty } |\phi |,
\]
where we used the fact that $v[\theta ](0,t)=0$ for all $t$ (due to the $P$-fold symmetry of $\theta$). The claim follows by applying the Gronwall inequality.
\end{proof}
We recall that, if $f\in C^{2 } (\R^2 ; \R^2)$ and $\widetilde{f}(x) \coloneqq f(0) + \nabla f (0) \cdot x$ denotes the first order Taylor expansion of $f$ at $0$ then
\eqnb\label{c2_error}
\| f - \widetilde{f} \|_{C^{1,\gamma }(\overline{B_R}) } \lec \| f \|_{C^2} R^{1 - \gamma } 
\eqne
for every $R>0$.

\section{Properties of oscillatory functions}\label{sec_osc_fcns}
In this section we prove two results concerning functions that involve high frequence oscillations.

\subsection{Decay at the origin of $P$-fold symmetric functions.}
Here we show the following
\begin{lemma}[Decay lemma]\label{decay}
Let $f\in C_0^\infty (\R^2 ;\R )$ be $P$-fold symmetric  and have zero mean over any circle centered at $0$. Then $|f(x)| \lec \| f \|_{C^P} |x|^P$ and $\mathcal{R}f \lec \| f \|_{C^P}  O(|x|^{P-1})$ as $x\to 0$.
\end{lemma}
Note that the lemma is sharp: take $f(x)\coloneqq \ee^{-|x|^2}x_2$. Then $f$ is $1$-fold symmetric and $f=O(|x|)$ as $x\to 0$ and
\[
\mathcal{R}_2 f (0) = \mathrm{pv} \int_{\R^2} \frac{y_2}{|y|^3} f(y) \d y=  \int_{\R^2} \frac{y_2^2}{|y|^3} \ee^{-|y|^2}  \d y= O(1)\hspace{1cm}\text{ as }x\to 0.
\]
\begin{proof}
Note that if $f$ is $P$-fold symmetric then its Fourier transform,
\[
\widehat{f} (\xi ) \coloneqq \frac1{2\pi } \int_{\R^2} f(x) \ee^{-ix\cdot \xi } \d x,
\]
is also $P$-fold symmetric.  In particular, $\widehat{f}$ has zero mean over any circle centered at $0$. Indeed, letting $\chi \in C_0^\infty (\R^2)$ be any radial function, we note that $\widehat{\chi}$ is radial as well, and so
\eqnb\label{int_fhat}
\int_{\R^2} \widehat{f} \chi = \int_{\R^2}  f  \widehat{\chi }=0
\eqne
using the zero-mean property of $f$. Since this is true for every such $\chi$, the zero-mean property of $\widehat{f}$ follows.

Consequently, using polar coordinates $(\rho,\beta )$ for $\xi \in \R^2$, we can expand $\widehat{f}$ into Fourier series in $\beta \in [0,2\pi)$ to obtain 
\[
\widehat{f}(\xi ) =  \sum_{n\geq 1} \left(F_n (\rho ) \cos (nP\beta ) + G_n (\rho ) \sin (nP \beta ) \right).
\]
For each $\alpha_1,\alpha_2\geq 0$  such that $|\alpha |=\alpha_1+\alpha_2 \leq P-1$ we have 
\[
\begin{split}
\p^\alpha f(0) &= i^{|\alpha |}\int_{\R^2} \xi_1^{\alpha_1}\xi_2^{\alpha_2} \widehat{f}(\xi ) \d \xi\\
&= i^{|\alpha |} \int_0^\infty \int_0^{2\pi } \sum_{n\geq 1} \left( F_n (\rho ) \rho^{|\alpha |+1} \cos (nP\beta ) + G_n(\rho )\rho^{|\alpha |+1} \sin (nP \beta )\right) \left( \cos \beta \right)^{\alpha_1} \left( \sin \beta \right)^{\alpha_2}  \d \beta \d \rho  \\
&=0,
\end{split}
\]
since $\left( \cos \beta \right)^{\alpha_1} \left( \sin \beta \right)^{\alpha_2} $ is a trigonometric polynomial of order $|\alpha|<P$, and is thus orthogonal to both $\cos (nP\beta )$ and $\sin (n P\beta )$.
Thus all partial derivatives of $f$ of order less than $P$ vanish at $x=0$, and thus Taylor's expansion shows that $f=O(|x|^P)$ as $x\to 0$.

As for $\mathcal{R}f$ we have $\widehat{\mathcal{R}f} (\xi ) = \frac{\xi }{|\xi |} \widehat{f} (\xi )$, and so a similar argument as above shows that all partial derivatives of $\mathcal{R}f$ of order less than  $P-1$ vanish at $x=0$, and the claim follows. 
\end{proof}

\subsection{Pseudo-velocity of an oscillating function}

Here we consider a function of the form 
\eqnb\label{oo}
\oo (x)=K^{-1} \lambda^{1-s}N^{-s}   {g(\lambda x)} {\sin(\lambda N x_{1}) \sin(\lambda N x_{2})},
\eqne 
 where $g\in C_c^\infty (B(0,1);[0,1] )$ is an arbitrary nontrivial function,  $\lambda$, $N$ satisfy $\lambda^{2-s}N^{1-s} = K \log N$ (recall~\eqref{lambdaN}),
and we quantify the approximation of the velocity field $v[\oo ]$ by
\eqnb\label{ov} \ov [ \oo ] (x,t) = \frac{g(\lambda x)}{\sqrt{2}K\lambda^{s-1}N^s } \begin{pmatrix}
-\sin (\lambda N x_1 )\cos (\lambda N x_2) \\
\cos (\lambda Nx_1) \sin (\lambda Nx_2)
\end{pmatrix}.
\eqne
That is, we show the following.
\begin{lem}\label{L01}
For every $M\geq 0$, $n\in \{ 2, \ldots , M \}$ we have that
\[
| (v[\oo  ] - \overline{v} [\oo ])(x)| \lec_M K^{-1}  \lambda^{1-s}N^{-s}\sum_{k=1}^M \frac{ | D^k g (\lambda x) |  }{ k!}  a^k +K^{-1} \lambda^{1-s} N^{-s} (1+ \| g \|_{C^{M+1}}) ( Na)^{-n} a\, \log a^{-1}
\]
whenever $a\in (0,1/2]$ is sufficiently small and $\lambda, N\geq 1$. Also
\[\begin{split}
| (\nabla v[\oo ] -\nabla  \overline{v} [\oo ])(x)| &\lec_M K^{-1} \lambda^{2-s}N^{1-s} \sum_{k=1}^M \frac{ | D^k g (\lambda x) |  + |D^{k+1} g (\lambda x)| }{ k!}  a^k \\
&\hspace{5cm}+K^{-1} \lambda^{2-s} N^{1-s} (1+ \| g \|_{C^{M+2}}) ( Na)^{-n} a\, \log a^{-1}.
\end{split}
\]
\end{lem}
In particular, given $\varepsilon >0$ we can apply the lemma with $M=n=2$, $a\coloneqq  N^{-1+\varepsilon }$ to obtain
\eqnb\label{vel_err_ck}
\| v [\oo ] - \ov [\oo ] \|_{\infty } \lec_\varepsilon  \| g \|_{C^{3}} K^{-1} \lambda^{1-s} N^{-1-s+\varepsilon} 
\eqne
\eqnb\label{vel_err_c1}
\| v [\oo ] - \ov [\oo ] \|_{C^1} \lec_\varepsilon \| g \|_{C^{4}} K^{-1}  \lambda^{2-s} N^{-s+\varepsilon } .
\eqne
Moreover, one can extend the lemma to any number $k$ of derivatives of $v [\oo ] - \ov [\oo ]$, and take $M=n=k+1$ to obtain
\eqnb\label{vel_err_ck1}
\| v [\oo ] - \ov [\oo ] \|_{C^k} \lec_{\varepsilon,k} \| g \|_{C^{k+2}} K^{-1}  \lambda^{1+k-s} N^{k-1-s+\varepsilon } .
\eqne
Using Lemma~\ref{L_holder_to_Hb} and H\"older interpolation we can thus deduce that
\eqnb\label{error_Hbeta}
\begin{split}
\| v [\oo ] - \ov [\oo ] \|_{H^\beta (\supp \, \oo )} &\lec \| v [\oo ] - \ov [\oo ] \|_{C^{1,\alpha }} \lambda^{\beta-2-\alpha }  \lec_\varepsilon K^{-1} \lambda^{\beta -s } N^{\alpha-s+\varepsilon}
\end{split}
\eqne
for $\beta \in (1,2)$, by taking any $\alpha \in (\beta -1,1)$.
\begin{proof}[Proof of Lemma~\ref{L01}.] Without loss of generality we assume that $K=1$. We note that\footnote{Recall that $\widehat{\sin (kx)}(\xi ) = \frac{i \pi}2 \left( \delta (\xi - k) - \delta (\xi + k) \right)$, $\widehat{\cos (kx)}(\xi ) = \frac{ \pi}2 \left( \delta (\xi - k) + \delta (\xi + k) \right)$ for each $k>0$ (in the sense of distributions) and so 
\[\begin{split}
\left( \mathcal{R}_1 (\sin (k x_1) \sin (kx_2) ) \right)^{\widehat{\mbox{}}}(\xi )&= -i \frac{\xi^\perp \pi }{ |\xi |} \cdot \frac{-\pi^2}{4} \left( \delta (\xi_1 - k) - \delta (\xi_1 + k) \right)\left( \delta (\xi_2 - k) - \delta (\xi_2 + k) \right)\\
&=  \frac{i \pi^2}{4\sqrt{2}} \left( \delta (\xi_1 - k) + \delta (\xi_1 + k) \right)\left( \delta (\xi_2 - k) - \delta (\xi_2 + k) \right)= \frac{1}{\sqrt{2}} \left( \cos (kx_1 ) \sin (kx_2) \right)^{\widehat{\mbox{}}}(\xi ),
\end{split}
\]
and similarly for $\mathcal{R}_2$, which gives \eqref{riesz_fact}.
}
\eqnb\label{riesz_fact}
\ov [\oo ] (x) = K^{-1} \lambda^{1-s} N^{-s} {g(\lambda x)} \mathcal{R}^\perp (\sin (\lambda N x_1) \sin (\lambda N x_2 ) ),
\eqne
where $\mathcal{R}^\perp \coloneqq (-\mathcal{R}_2,\mathcal{R}_1)$, and so
\[
u\coloneqq v[\oo ] - \overline{v} [\oo ] = \lambda^{1-s}N^{-s} \int_{\R^2} \frac{(\lambda x-y)^\perp (g(y) - g(\lambda x) )}{|\lambda x-y|^3} \sin ( N y_1) \sin ( N y_2) \d y.
\]

We write $u=: u_{\rm in} + u_{\rm out}$,
\[\begin{split}
u_{\rm in} (x)&\coloneqq \lambda^{1-s} N^{-s} \int_{\R^2} \frac{(\lambda x-y)^\perp (g (y)-g(\lambda x)) }{|\lambda x-y|^3} \chi (\lambda x-y) \sin( Ny_1)\sin(Ny_2)\d y, \\
u_{\rm out} (x) &\coloneqq \lambda^{1-s}N^{-s} \int_{\R^2} \frac{(\lambda x-y)^\perp (g (y)-g(\lambda x))}{|\lambda x-y|^3} (1-\chi (\lambda x-y)) \sin( Ny_1)\sin( Ny_2)\d y ,
\end{split}
\]
where $\chi \in C_c^\infty (\R^2 )$ is such that $\chi =1$ on $B_a$ and $\chi=0$ on $B_{2a}^c$. For $u_{\rm in}$ we apply Taylor expansion of order $M$ to write $u_{\rm in}=u_{\rm in,1}+u_{\rm in,2}$, where
\[
\begin{split}
u_{\rm in,1} (x) &\coloneqq \lambda^{1-s}N^{-s} \int_{\R^2} \frac{(\lambda x-y)^\perp  \sum_{k=1}^M \frac{D^k g (\lambda x)}{k!} (y-\lambda x)^k }{|\lambda x-y|^3} \chi (\lambda x-y) \sin(Ny_1)\sin( Ny_2)\d y,  \\
u_{\rm in,2} (x) &\coloneqq \lambda^{1-s}N^{-s} \int_{\R^2} \frac{(\lambda x-y)^\perp \left(g (y)- \sum_{k=0}^M \frac{D^k g (\lambda x)}{k!} (y-\lambda x)^k \right) }{|\lambda x-y|^3} \chi (\lambda x-y) \sin( Ny_1)\sin( Ny_2)\d y  .
\end{split}
\]
For $u_{\rm in,1}$ we obtain
\[
\begin{split}
u_{\rm in,1} (x) &= \lambda^{1-s}N^{-s}\sum_{k=1}^M \frac{D^k g (\lambda x)}{k!} \colon  \int_{\R^2} \frac{(y-\lambda x)^k (\lambda x-y)^\perp }{|\lambda x-y|^3}\chi (\lambda x-y) \sin ( Ny_1)\sin( Ny_2)\d y\\
&\lec_M \lambda^{1-s}N^{-s} \sum_{k=1}^M \frac{ | D^k g (\lambda x) |  }{ k!}  a^k  .
\end{split}
\]
As for $u_{\rm in,2}$, we integrate $n$ times by parts in $y_1$ (recall $n\leq M$) to obtain
\[
u_{\rm in,2} (x) = \lambda^{1-s}N^{-n-s}(-1)^{c_n} \int_{\R^2} \sum_{|\alpha | + |\beta | + |\gamma | = n} C_{\alpha, \beta , \gamma } \cdot J_{\alpha,\beta , \gamma } \cdot    ( \sin \text{ or } \cos )( Ny_1)\sin( Ny_2)\d y,
\]
where $c_n \in \N$, $C_{\alpha,\beta,\gamma} >0$, 
\[
J_{\alpha,\beta , \gamma }(x,y) \coloneqq D^\alpha \frac{(\lambda x-y)^\perp  }{|x-y|^3}D^\beta \left(g (y)- \sum_{k=1}^M \frac{D^k g (\lambda x)}{k!} (y-\lambda x)^k \right) D^\gamma  \chi (\lambda x-y).
\]
Note that 
\[
| J_{\alpha, \beta , \gamma } ( x,y) | \lec_M |\lambda x-y|^{-|\alpha |-2} \| g \|_{C^{M+1}}  |\lambda x-y |^{M+1-|\beta |}  a^{-|\gamma |},
\]
and so
\[
\begin{split}
|u_{\rm in,2} (x)| &\lec_M \lambda^{1-s}N^{-s} \| g \|_{C^{M+1}} ( Na)^{-n}a,
\end{split}
\]
where we also used the fact that $a<1$.

As for $u_{\rm out}$, we similarly integrate by parts in $y_1$  $n$ times to obtain
 \[
u_{\rm out} (x) = \lambda^{1-s}N^{-n-s} (-1)^{c_n} \int_{\R^2} \sum_{|\alpha | + |\beta | + |\gamma | = n} C_{\alpha, \beta , \gamma }\cdot  K_{\alpha,\beta , \gamma } \cdot   ( \sin \text{ or } \cos )( Ny_1)\sin( Ny_2)\d y,
\]
where
\[
K_{\alpha,\beta , \gamma }(x,y) \coloneqq D^\alpha \frac{(\lambda x-y)^\perp  }{|\lambda x-y|^3}D^\beta (g (y) -g(\lambda x ) )D^\gamma (1- \chi (\lambda x-y)).
\]

We have that
\[
\begin{split}
\int_{\R^2} \left| K_{\alpha,\beta,\gamma } \right| &\lec \begin{cases} 
\| g \|_{C^1}\int_{|\lambda x-y|\geq a} |\lambda x - y |^{-|\alpha |-1}  \d y \qquad &|\gamma|=|\beta| =0,\\
\| g \|_{C^{|\beta |}}\int_{|\lambda x-y|\geq a \cap \supp\, g} |\lambda x - y |^{-|\alpha |-2}  \d y \qquad &|\gamma|=0,|\beta| \geq 1,\\
\| g \|_{C^1 }a^{-|\gamma |}\int_{|\lambda x-y|\sim a} |\lambda x - y |^{-|\alpha |-1}  \d y \qquad &|\gamma|\geq 1,|\beta| =0,\\
\| g \|_{C^{|\beta |} }\,a^{-|\gamma |}\int_{|\lambda x-y|\sim a} |\lambda x - y |^{-|\alpha |-2}  \d y \qquad &|\gamma|,|\beta |\geq 1
\end{cases} \\
&\lec 
a^{-|\alpha|-|\gamma |-|\beta |+1} \log a^{-1}
\end{split}
\]
for $n=|\alpha | + |\beta | + |\gamma | \geq 2$. Thus
\[\begin{split}
|u_{\rm out} (x) | &\lec_n \lambda^{1-s}N^{-n-s}\| g \|_{C^n}  \sum_{|\alpha | + |\beta | + |\gamma | = n}  a^{-|\gamma |-|\alpha |} \log a^{-1} (a   + 1_{|\beta |\geq 1} )\\
& \lec_n N^{-s}\| g \|_{C^n} ( Na)^{-n} a\, \log a^{-1},
\end{split}
\]
as required.

As for the gradient note that 
\[\begin{split}
\nabla u &= \lambda^{1-s}N^{-s}\int_{\R^2} \frac{(\lambda x-y)^\perp}{|\lambda x-y|^3} ( \na g(y)  - \na g(\lambda x) )\sin( Ny_1) \sin( Ny_2) \d y\\
& \hspace{1cm}+ \lambda^{1-s}N^{1-s}\int_{\R^2} \frac{(\lambda x-y)^\perp}{|\lambda x-y|^3} (  g(y)  -  g(\lambda x) )\sin/\cos ( Ny_1) \sin/\cos ( Ny_2)  \d y
\end{split}
\]
Similarly as above we now decompose $\nabla u = \nabla u_{\rm in} + \nabla u_{\rm out}$ to obtain an analogous estimate, except with an extra power of $\lambda N$.
\end{proof}

\begin{corollary}\label{cor_vel_err}
Suppose that 
\[
\oo (x) = \lambda^{1-s} N^{-s} K^{-1} g(\lambda x) \sin (\lambda N x_1)
\]
and that $L(x) \coloneqq (d_1 x_1,d_2 x_2)+b$, with $0<d_2<d_1$, $d_1\geq 1$, $b\in \R^2$. Then
\eqnb\label{vel_err_cor}
\| v[\oo \circ L ] - \ov [\oo \circ L ] \|_{C^{k,\alpha}} \lec_{\varepsilon } \| g \|_{C^{k+4}} d_1^{2k+3+\alpha + \varepsilon  } K^{-1} \lambda^{1+k+\alpha -s} N^{-1+k+\alpha -s+\varepsilon }
\eqne
for all $k\geq 0$, $\alpha \in [0,1)$, where 
\[
\ov [\oo \circ L ] \coloneqq \lambda^{1-s} N^{-s} K^{-1} g(\lambda L(x) ) \begin{pmatrix}
0 \\ \cos \left( \lambda N L(x)_1\right)  
\end{pmatrix}.
\]
\end{corollary}
\begin{proof}
If $L=\mathrm{id}$ then the claim follows as in Lemma~\ref{L01}, by integrating by parts in $x_1$ (recall in the proof above we only used integration by parts with respect to one of the variables), and then taking $M=n=2$, $a=N^{-1+\varepsilon }$. This in particular gives $\| g \|_{C^{k+3}}$ norm on the right-hand side (rather than $C^{k+2}$, since we bound the $C^{k,\alpha}$ norm by the $C^{k+1}$ norm).

If $L\ne \mathrm{id}$, then, without loss of generality we assume that $b=0$ (as both the left-hand side and the right-hand side of \eqref{vel_err_cor} are invariant with respect to translations). We set $f(x) \coloneqq g( d_1 x_1, d_2 x_2)$ and $N' \coloneqq d_1 N$. We see that
\[
\oo =d_1^s  \lambda^{1-s} (N')^{-s} K^{-1} f(\lambda x) \sin (\lambda N' x_1).
\]
Thus, we can  apply \eqref{vel_err_cor} (with $L=\mathrm{id}$) to obtain
\[
\begin{split}
\| v [\oo \circ L ] - \ov [\oo \circ L ] \|_{C^{k,\alpha }} &\lec_\varepsilon d_1^s \| f \|_{C^{k+4}} K^{-1} \lambda^{1+k+\alpha -s } (N')^{-1+k+\alpha -s + \varepsilon } \\
&\leq d_1^{2k+3+\alpha +\varepsilon } \| g \|_{C^{k+4}} K^{-1} \lambda^{1+k+\alpha -s} N^{-1+k+\alpha -s + \varepsilon },
\end{split}
\]
as required.
\end{proof}

\section{Norm inflation}\label{sec_norm_infl}

Here we prove Theorem~\ref{T00}. We will suppose that the solution $\theta$ to the SQG equation consists of the background part $\tb$ and the perturbation part $\tp$,
so that 
\[
\theta = \tb + \tp.
\]
As for the background part $\tb$, we will use the \emph{pseudosolution for the background}
\eqnb\label{pseudo_bg}
\otb (x) \coloneqq \sum_{m=0}^{P-1} \oo (R_{-m}( x_1-\lambda^{-1/2},x_2)) =: \sum_{m=0}^{P-1} \otb_m ,
\eqne
where $\oo$ is defined by \eqref{oo}, i.e. $\oo (x)=K^{-1} \lambda^{1-s}N^{-s}   {g(\lambda x)} {\sin(\lambda N x_{1}) \sin(\lambda N x_{2})}$, and we denote by
\[
R_m (x_1,x_2) \coloneqq \left(\cos \frac{2\pi m}P x_1 - \sin \frac{2\pi m}P x_2, \sin \frac{2\pi m}P x_1 + \cos \frac{2\pi m}P x_2\right), \qquad  m\in \Z,
\]
the rotation (counterclockwise) by $2\pi m/P$. We also define the \emph{pseudovelocity for the background},
\[
\ov [\otb ](x) \coloneqq \sum_{m=0}^{P-1} R_m \ov [\oo  ](R_{-m}x).
\]

Note that 
\eqnb\label{otb_Cka_est}
\| \otb \|_{C^{k,\alpha}} \sim_{k,\alpha} (\lambda N )^{k-1+\alpha } \log N 
\eqne
for all $k\geq 0$, $\alpha \in [0,1)$. We denote by $\tb$ the local-in-time strong solution to the SQG equation with initial data $\otb $. Since rotational symmetries are preserved by the SQG flow, we note that $\tb (\cdot ,t )$ is $P$-fold symmetric for each $t$.

\subsection{The background flow}

In this section we verify that $\tb (t)$ can be approximated, for a short time, by $\otb$. Namely, we set 
\[ 
\ttb \coloneqq \tb - \otb
\]
(recall \eqref{pseudo_bg} that $\otb$ is a time-independent function), and we let $\alpha >0$ be any number such that
\[
\alpha <2-s.
\]
In order to obtain a sufficiently strong control of the error $\ttb$, we will estimate its $C^{1,\alpha}$ norm on a short time interval $[0,T]$. To be precise, given $P\geq 1$ (the periodicity in the angular variable) and $K\geq 1$ (the smallness parameter of the data \eqref{pseudo_bg}) we let $\varepsilon \in (0,1/10] $ and $T$ be sufficiently small so that 
\eqnb\label{choice_eps_T}
\alpha <(2-s) (1-\varepsilon - 3\kappa  T),
\eqne
where $\kappa = \kappa  (P,K,\alpha )$ is a constant defined in \eqref{bg_c1alpha} below. With such choices of $\varepsilon , T$ we prove the following estimates on the background error $\ttb$.
\begin{lemma}[Control of the background flow]\label{L_bg}
For all sufficiently large $\lambda \geq 1$ (dependent on $P,K$) $\tb $ exists and remains smooth until $T$ and
\eqnb\label{bg_ests}
\begin{split}
\| \ttb \|_\infty &\lec_{K,P} \lambda^{-1} N^{-2+\varepsilon + ct} { (\log N)^2t } ,\\
\| \ttb \|_{C^1} &\lec_{K,P}  N^{-1+\varepsilon + 2ct } {(\log N)^3 } ,\\
\| \ttb \|_{C^{1,\alpha}} &\lec_{K,P}    N^{\frac{\alpha(s-1)}{2-s} + \alpha - 1 +\varepsilon +3cT } t (\log N)^{4+\alpha/(2-s)},\\
\|  \ttb \|_{C^{k,a }} &\lec_{K,P,a,k}  \lambda^{k-1+a} N^{k-2+a + \varepsilon +3ct } (\log N )^4,\\
\| \ttb \|_{H^\beta } &\lec_{K,P}  \lambda^{\beta -s} N^{\beta -s-1 +\varepsilon +cT}  (\log N)^2
\end{split}
\eqne
for all $t\in [0,T]$, $k\geq 0$, $a\in (0,1)$, $\beta \in (1,2)$.
\end{lemma}
\begin{proof}
We first note that $\ttb$ satisfies the PDE
\eqnb\label{eq_ttb}
\begin{split}
-\p_t \ttb &=  v[\tb  ] \cdot \na \tb \\
&=v[\tb  ] \cdot \na \tb - \ov [\otb ] \cdot \na \otb +G \\
&= v[\otb ] \cdot \na \ttb +v[\ttb ] \cdot \nabla \ttb+  \left( v[\otb ]- \ov [\otb ] \right) \cdot \nabla \otb + v [\ttb ] \cdot \na \otb + G 
\end{split}
\eqne
with homogeneous initial data, where
\[
G\coloneqq \ov [\otb ] \cdot \nabla \otb.
\]
Due to the separation of the pieces in \eqref{pseudo_bg} $G$ enjoys the same estimates as in \eqref{cancel_C1}--\eqref{cancel_Hbeta}. Namely we see that $\| G \|_{C^k} \lec \lambda^{k -1} N^{k -2} (\log N)^2$ for all $k\geq 0$, and so, by H\"older interpolation and Lemma~\ref{L_holder_to_Hb},
\eqnb\label{G_ests}
\begin{split}
\| G \|_{C^{k,\alpha}} &\lec  \lambda^{k+\alpha -1} N^{k+\alpha -2} (\log N)^2,\\
\|G \|_{\dot H^\beta} &\lec_P \| \ov [\oo ] \cdot \na \oo \|_{C^{1,\alpha }} \lambda^{\beta-\alpha -2}\lec \lambda^{\beta -2} N^{\alpha -1} (\log N)^2
\end{split}
\eqne
for $k\geq 0$, $\alpha \in [0,1)$, $\beta \leq k+\alpha $. As for the third term on the right-hand side of \eqref{eq_ttb}, we recall from Lemma~\ref{L01} that
\eqnb\label{vel_est1}
\| v[\otb ] - \ov [\otb ] \|_{C^{k,\alpha}} \lec_P K^{-1} \lambda^{1+k+\alpha -s} N^{-1+k+\alpha -s +\varepsilon }= \lambda^{-1+k+\alpha  }N^{-2+k+\alpha +\varepsilon } \log N
\eqne
for $k\geq 0$, $\alpha \in [0,1)$. We let $T' \in [0,T]$ be the largest time such that $\| \ttb \|_{C^{1,\alpha }} \leq 1$ for $t\leq T'$. In particular, we can use \eqref{vel_est} to obtain
\eqnb\label{vel_esttt}
\| v[ \ttb ] \|_{L^\infty } \lec \| \ttb \|_{L^\infty }, \qquad \| v[ \ttb ] \|_{C^1 } \lec \| \ttb \|_{C^1 }
\eqne
for $t\in [0,T']$, and we have 
\[
\begin{split}
\p_t \| \ttb \|_{\infty} &\lec   \left( \| v[\otb ] - \ov [\otb ] \|_{\infty } + \| v [\ttb ] \|_\infty \right) \| \na \otb \|_\infty +  \| G \|_\infty\\
&\lec \left( C_{K,P} \lambda^{-1} N^{-2+\varepsilon }\log N + \| \ttb \|_{\infty}  \right) \log N + \lambda^{-1}N^{-2} (\log N)^2\\
&\leq \| \ttb \|_\infty \log N + C_{K,P}\lambda^{-1}N^{-2+\varepsilon } (\log N)^2 ,
\end{split}
\]
where we used \eqref{G_ests}, \eqref{vel_est1} and \eqref{otb_Cka_est}. Hence (recalling the ODE fact~\eqref{ode_fact}),
\eqnb\label{ttb_infty}
\| \ttb \|_\infty \lec_{K,P} \lambda^{-1} N^{-2+\varepsilon + ct} { (\log N)^2t } ,
\eqne
and consequently
\eqnb\label{v_001}
\| v[\tb ] - \ov [\otb ]\|_{L^\infty } \leq \| v [\ttb ] \|_{L^\infty} + \| v [\otb ] - \ov [\otb ] \|_{L^\infty } \lec_{K,P} \lambda^{-1} N^{-2+\varepsilon +ct } (\log N)^2.
\eqne

As for the $C^1$ norm,
\[
\begin{split}
\p_t \| \ttb \|_{C^1 } &\lec  \left( \| \ov [\otb ]\|_{C^1} + \| v [\otb ] - \ov [\otb ] \|_{C^1} + \| v [\ttb ] \|_{C^1}\right) \| \ttb \|_{C^1}    + \|  v [\otb ] -  \ov [\otb ] \|_{C^1} \| \na \otb \|_\infty \\
&\hspace{0.5cm}+\|  v [\otb ] -  \ov [\otb ] \|_{\infty } \| \na \otb \|_{C^1}  +  \| v [\ttb ] \|_\infty \| \otb \|_{C^2}+ \| v [\ttb ] \|_{C^1} \| \otb \|_{C^1} + \| G \|_{C^1}\\
&\lec   \| \ttb \|_{C^1} \left( \log N  + C_{K,P}N^{-1+\varepsilon } \log N + \| \ttb \|_{C^1}\right)  +N^{-1+\varepsilon } (\log N )^2\\
&\hspace{0.5cm} +  \| \ttb \|_{L^\infty} \lambda N \log N +\| \ttb \|_{C^1} \log N   \\
&\lec \| \ttb \|_{C^1} \log N + C_{K,P}(\log N)^3 N^{-1+\varepsilon+ct} ,
\end{split}
\]
where we used \eqref{vel_esttt}, \eqref{ttb_infty}, \eqref{otb_Cka_est}, \eqref{G_ests} and \eqref{vel_est1}. Thus 
\eqnb\label{ThetaC1}
\| \ttb \|_{C^1} \lec_{K,P} N^{-1+\varepsilon + 2ct }{(\log N)^3 t} \lec N^{-1+\varepsilon + 2ct } {(\log N)^3 } 
\eqne
for $t\leq T'$, and consequently
\eqnb\label{vWC1}
\| v [\ttb ]\|_{C^1} \lec_{K,P} N^{-1+\varepsilon + 2ct } {(\log N)^3 } 
\eqne
for $t<T'\leq 1$. This and \eqref{vel_err_c1} imply that
\eqnb\label{vel_bg_c1}
\| v [\tb ] - \ov [\otb ] \|_{C^1} \lec_{K,P}  N^{-1+\varepsilon + 2ct } {(\log N)^3 }
\eqne
for such $t$.

For the $C^{1,\alpha}$ norm estimate, we obtain, as in Wu \cite{Wu}, 
\[\begin{split}
\frac{\d }{\d t}\| \ttb \|_{C^{1,\alpha}}& \lec \left( \| \ov [\otb ]\|_{C^{1,\alpha} } + \| v [\otb ] - \ov [\otb ] \|_{C^{1,\alpha}} + \| v [\ttb ] \|_{C^{1,\alpha}}\right) \| \ttb \|_{C^1}\\
&\hspace{0.5cm}+\left( \| \ov [\otb ]\|_{C^1} + \| v [\otb ] - \ov [\otb ] \|_{C^1 } + \| v [\ttb ] \|_{C^1 }\right) \| \ttb \|_{C^{1,\alpha}} \\
& \hspace{0.5cm} + \|  v [\otb ] -  \ov [\otb ] \|_{C^{1,\alpha}} \|  \otb \|_{C^1} +\|  v [\otb ] -  \ov [\otb ] \|_{\infty } \|  \otb \|_{C^{2,\alpha}}  +  \| v [\ttb ] \|_\infty \| \otb \|_{C^{2,\alpha}}\\
&\hspace{0.5cm}+ \| v [\ttb ] \|_{C^{1,\alpha }} \| \otb \|_{C^1} + \| G \|_{C^{1,\alpha}}\\
&\lec \left( (\lambda N)^{\alpha}\log N + C_{K,P}\lambda^{\alpha} N^{-1+\alpha + \varepsilon } \log N+ \| \ttb  \|_{C^{1,\alpha}}\right)  N^{-1+\varepsilon + 2ct } (\log N)^3   \\
&\hspace{0.5cm}+\left( \log N +C_{K,P}  N^{-1+\varepsilon + 2ct } {(\log N)^3 } + C_{K,P}  N^{-1+\varepsilon + 2ct } {(\log N)^3 } \right) \| \ttb \|_{C^{1,\alpha}} \\
& \hspace{0.5cm} + C_{K,P}\lambda^{\alpha} N^{-1+\alpha + \varepsilon } (\log N)^2  + C_{K,P}\lambda^{\alpha }N^{-1+\alpha + \varepsilon + ct} (\log N )^2  + \| \ttb  \|_{C^{1,\alpha }} \log N \\
&\hspace{0.5cm}+ \lambda^{\alpha } N^{-1+\alpha } (\log N )^2 \\
&\leq \| \ttb \|_{C^{1,\alpha } } \log N +C_{K,P}\lambda^{\alpha} N^{-1+\alpha + \varepsilon +2ct } (\log N )^4.
\end{split}
\]
Thus there exists $\kappa >0$ such that
\eqnb\label{bg_c1alpha}
\| \ttb \|_{C^{1,\alpha}} \lec_{K,P}  \lambda^\alpha N^{\alpha -1+\varepsilon +3\kappa  t} t (\log N)^4 \lec_K  N^{\frac{\alpha(s-1)}{2-s} + \alpha - 1 +\varepsilon +3\kappa 0 T } t (\log N)^{4+\alpha/(2-s)}
\eqne
for $t\leq T'$. 
Note that, by \eqref{choice_eps_T} the power of $N$ in \eqref{bg_c1alpha} is negative, which shows that for all sufficiently large $\lambda$, \eqref{bg_c1alpha} holds in fact for all $t\in [0,T]$. Thus $T'=T$, which in particular proves the first three estimates in \eqref{bg_ests}.

A similar computation as for the $C^{1,\alpha}$ estimate above gives 
\[
\|  \ttb \|_{C^{k,\alpha }} \lec_{K,P}  \lambda^{k-1+\alpha} N^{k-2+\alpha + \varepsilon +3ct } (\log N )^{C_{k,\alpha}}
\]
for $k\geq 0$, $\alpha \in (0,1)$, $t\in [0,T]$, as required.

Finally we can use the control of $C^{1,\alpha}$ to obtain smallness of $\ttb$ in $H^{\beta }$; we have 
\[
\begin{split}
\p_t \| \ttb \|_{H^\beta } &\lec  \left( \| \ov [\otb ]\|_{L^\infty } + \| v [\otb ] - \ov [\otb ] \|_{L^\infty } + \| v [\ttb ] \|_{L^\infty }\right)  \| \ttb \|_{H^\beta}  \\
&\hspace{0.5cm}+ \left( \| v [\otb ]\|_{H^\beta } + \| v [\ttb ] \|_{H^\beta }\right)  \| \ttb \|_{C^1 } \\
&\hspace{0.5cm}+ \| (v  [\otb ] - \ov [\otb ] )\cdot \nabla \otb \|_{H^\beta} +   \| v [\ttb] \|_\infty \| \otb \|_{H^{\beta+1}} + \| v [\ttb ] \|_{H^\beta} \| \otb \|_{C^1}+ \| G \|_{H^\beta }\\
&\lec \lambda^{-1}N^{-2+\varepsilon + ct} (\log N )^2 \| \ttb \|_{H^\beta} \\
&\hspace{0.5cm}+C_{K,P}\left( (\lambda N )^{\beta -s}+ \lambda^{-1}N^{-2+\varepsilon +ct} (\log N)^2 \right) N^{-1+\varepsilon +2ct } (\log N)^3 +  C_{K,P}\lambda^{\beta -2} N^{-1+\alpha' +\varepsilon } (\log N)^2 \\
&\hspace{0.5cm}+  C_{K,P}\lambda^{\beta -s} N^{\beta - s -1 + \varepsilon + ct} (\log N )^2 + \| \ttb \|_{H^\beta } \log N + C_{K,P} \lambda^{\beta -2} N^{\alpha'-1} (\log N)^2\\
&\lec  \| \ttb \|_{H^\beta } \log N +  C_{K,P}  \lambda^{\beta -s} N^{\beta - s -1 + \varepsilon + 2ct} (\log N )^2 
\end{split}
\]
for  $\beta \in \{ 1,2\} $,  where we used \eqref{G_ests} and \eqref{error_Hbeta} in the second inequality. Hence
\[
\| \ttb \|_{H^\beta } \lec_{K,P} \lambda^{\beta -s} N^{\beta - s -1 + \varepsilon + 3ct} (\log N )^2 
\]
for  $\beta \in \{ 1,2\} $, $t\in [0,T]$ and applying Sobolev interpolation this gives the $H^\beta$ estimate claimed in \eqref{bg_ests}.
\end{proof}

\subsection{The perturbation}

For the perturbation $\tp$ we consider initial conditions of the form
\eqnb\label{pert_data}
\tp (x,0) = \sum_{m=0}^{P-1} \omega_p \circ R_{-m}(x-(\lambda^{-1/2},0)) =: \sum_{m=0}^{P-1} \tp_m (x,0),
\eqne
where 
\eqnb\label{pert_def}
\omega_{p}(x)\coloneqq K^{-1}\lt^{1-s} \nt^{-s} {g(\lt x)}  \sin(\lt \nt  x_{1}).\eqne
Let us denote 
\[
z \coloneqq \left( \lambda^{-1/2 },0\right)
\]
for brevity. We denote by $\tp (x,t)$ the \emph{pertubation}; namely the function such that
\[
\theta \coloneqq \tb+ \tp
\]
is a solution to the SQG equation \eqref{sqg} with initial condition $\tb (\cdot ,0) + \tp(\cdot ,0)$.\\

We denote by $\eta$ the trajectory of the background flow, namely the solution to the ODE problem
\[
\p_t \eta (x,t) = v[\tb ](\eta(x,t),t), \qquad \eta(x,0) =x.
\]
Noting that $\oo$ (recall \eqref{oo}) has a hyperbolic point at $x=0$ we see that $\tb $ has $P$ hyperbolic points of $\tb$ located near $\eta(R_m(z),t)$, $m=0,\ldots , P-1$. We note that they are not necessarily located exactly at any of these points, say at $\eta (z,t)$ (i.e. $\eta(R_0(z,t))$), due to the influence of the remaining pieces $\otb_m \circ \eta^{-1}$, $m=1,\ldots , P-1$, via the velocity field. However, due to the $P$-fold symmetry, these hyperbolic points are located on the lines passing through $0$ and $\eta(R_m(z,t))$, $m=0,\ldots , P-1$. We emphasize that we do not need to control precisely the locations of these hyperbolic points, as the deformation due to a hyperbolic point persist in its neighbourhood, but we will need to ensure that each piece of the perturbation \eqref{pert_data} moves very little during time interval $[0,T]$.  (Recall that $T$ was fixed in \eqref{choice_eps_T}.)

For this, we first ensure that the trajectories $\eta $ originating from $R_m (z)$, $m=0,\ldots , P-1$, do not move away too much from the initial points. We achieve this by making use of the pseudosolution $\otb$ of the background; namely by noting that $\ov [\otb ]$ has hyperbolic points \emph{exactly} at $R_m (z)$, $m=0,\ldots , P-1$. At these points $\ov [\otb ]$ vanishes, which enables us to derive a Gronwall estimate for $|\eta (z,t) - z|$, see \eqref{eta_gronwall} below. Secondly, we will approximate $\tp$ by a pseudosolution $\otp$ which is deformed according to the 1st order Taylor expansion $\widetilde{v}$ of $v[\tb ]$ centered at $\eta (z,t)$. This way we keep track of the stretching due to $v[\tb ]$, and also control $\| \widetilde{v} - v[\tb ]\|_{C^{1,a}}$ for some $a>0$, see \eqref{vtilde_err_C1gamma}.
We note that $\tb (\cdot , t)$ is odd-odd and $P$-fold symmetric, since the SQG equation preserves these symmetries. Moreover $\tb (\cdot , t)$ is odd with respect to each of the lines $\{ (x,y) \in \R^2 \colon \arctan \frac{y}{x} = 2\pi m/P \}$, such that $m\in \Z$ or $m+1/2 \in \Z$. Indeed, the case of integer $m$ is clear by the rotations $R_m$ appearing in the definitions of the definitions of the initial data $\theta (\cdot ,0)$ (recall~\eqref{pseudo_bg} and \eqref{pert_data}), while the case of $m+1/2\in \Z$ follows from the odd symmetry of each piece with respect to its axis, namely that $\tb_0(x,y,0)=-\tb_0(x,-y,0)$ and $\tp_0(x,y,0)=-\tp_0(x,-y,0)$. In particular, defining the sectors
\[
S_m \coloneqq R_m (S_0), \quad \text{ where }\quad S_0\coloneqq \left\lbrace (x,y) \in \R^2 \colon \arctan\, \frac{y}{x} \in \left( -\frac{\pi}{P},\frac{\pi}P \right) \right\rbrace
\]
and $m=0,\ldots , P-1$, we see that
\eqnb\label{v_bdry_sectors}
n\cdot v[\theta ](x ,t)=0 \qquad \text{ for }\quad x\in \p S_m, m=0,\ldots , P-1.
\eqne

In other words, we can write 
\[
\tp = \sum_{m=0}^{P-1} \tp_m,
\]
where $\supp\,\tp_m \subset S_m $. \\

We now approximate $v [\tb ]$ on $\supp \, \tp_m $ by the hyperbolic flow near $\eta (R_m(z),t)$; namely we let
\[
A(t) \coloneqq \int_0^t \nabla v[\tb ](\eta(z,s),s) \d s,\qquad b(t) \coloneqq \int_0^t \ee^{-A(s)}\left( v[\tb ](\eta (z,s),s) -\nabla v[\tb ](\eta(z,s),s)  \eta(z,s) \right) \d s 
\]
and
\[
\tv (x,t) \coloneqq   b'(t) + A'(t) x ,
\]
and we estimate the evolution of $\tp$ by the solution $\otp$ to

\[\p_t \otp+\tv \cdot \nabla \otp =0
\]
with initial data $\otp(x,0)=\tp (x,0)$, i.e.
\eqnb\label{def_otp}
\otp (x,t) \coloneqq  \sum_{m=0}^{P-1} \omega_p  \circ R_{-m}  (\widetilde{\eta }^{-1} (x,t)) .
\eqne
where
\eqnb\label{teta_def}
\widetilde{\eta }^{-1} (x,t) = \ee^{-A(t) }x-b(t) 
\eqne
is the inverse of the Lagrangian trajectory of $\tv$.
\begin{lemma}[Structure of the deformation matrix]\label{L_A}
We have that 
\eqnb\label{size_of_A}
A(t) = \frac{1}{\sqrt{2}}\begin{pmatrix}
-\log N & 0 \\
0 & \log N
\end{pmatrix} t + \begin{pmatrix}
t & 0 \\
0 & t
\end{pmatrix}  O(N^{\alpha (-1+2\varepsilon + 2 ct )} \log N )
\eqne
as $N\to \infty$, uniformly with respect to $t\in [0,T]$, where $c>0$ is a constant such that $c\ll \kappa $.
\end{lemma}
(Recall \eqref{lambdaN} that $\lambda^{2-s}N^{1-s} = K \log N$, and that $T>0$ is determined in \eqref{choice_eps_T} above.)
Note that the last claim of the lemma, together with the fact that $\varepsilon \in (0,1/10]$ implies that the second term on the right-hand side of \eqref{size_of_A} is negligible compared to the first term, on $[0,T]$.
\begin{proof}
We first note that, due to the odd-odd symmetry we have that $v_2[\tb ]=0$ on the $x_1$ axis and that $v_1 [\tb ]$ is even in $x_2$. Thus
\[
\p_2 v_1[\tb ] = \p_1 v_2 [\tb ] =0\qquad \text{ on the }x_1 \text{ axis}
\]
for all $t\in [0,T]$ (due to the odd-odd symmetry), and so in particular at $(\eta (z,t),t )$ ($t\in [0,T]$). This implies that the deformation matrix $A$ is diagonal. In order to obtain \eqref{size_of_A}, we note that
\eqnb\label{eta_gronwall}
\begin{split}
 \frac{\d }{\d t} |\eta (z,t)-z| &\leq  |v [\tb ] (\eta (z,t),t) | \\
 &\leq \|v [\tb ] - \ov [\otb ]\|_\infty + |\ov [\otb ] (\eta (z,t),t) -\ov [\otb ](z,t) |\\
 & \leq C_{K,P} \lambda^{-1} N^{-2+\varepsilon +ct } (\log N)^2  + \| \ov [\otb ] \|_{C^{1}} | \eta(z,t)-z|\\
 & \lec  C_{K,P} \lambda^{-1} N^{-2+\varepsilon +ct } (\log N )^2   + \log N  | \eta(z,t)-z|,
 \end{split}
\eqne
where we used \eqref{v_001} in the third line, and so the ODE fact \eqref{ode_fact} implies that 
 \[
 |\eta (z,t)-z| \leq C_{K,P} \lambda^{-1} N^{-2+\varepsilon +2ct } .
 \]
 Thus
\[\begin{split}
\left| \na v [\tb ] (\eta (z,t),t) - \na \ov [\otb ] (z,t) \right| &\leq \left| \na v [\tb ] (\eta (z,t),t)- \na v [\tb ] (z,t)\right| + \| v [\tb ] - \ov [\otb ] \|_{C^1} \\
&\lec_{K,P} \| v[\tb ] \|_{C^{1,\alpha }} |\eta (z,t) - z|^\alpha + N^{-1+\varepsilon +2ct} (\log N)^3 \\
&\lec_{K,P} N^{\alpha(-1+\varepsilon +2ct )}   \log N ,
\end{split}
\]
where we used \eqref{vel_est1} and \eqref{bg_c1alpha} to get 
\[\begin{split}
 \| v[\tb ] \|_{C^{1,\alpha }} &\leq \| \ov [\otb ] \|_{C^{1,\alpha }}+ \| v [\otb ] - \ov [\otb ] \|_{C^{1,\alpha }} + \| v[\ttb ] \|_{C^{1,\alpha }} \\
 &\lec_{K,P} (\lambda N)^\alpha \log N +  \lambda^\alpha N^{-1+\alpha + \varepsilon } \log N +  \lambda^\alpha N^{-1+\alpha + \varepsilon + 3ct } (\log N)^4\\
 &\lec_{K,P} (\lambda N)^\alpha \log N .
 \end{split}
 \]
Thus \eqref{size_of_A} follows, since $\nabla \ov[\otb ](z,t)$ equals to the first term on the right-hand side of \eqref{size_of_A}. Note also that we can assume that $c\ll \kappa $, since both constants arise from similar estimates, and we could have chosen $\kappa $ (in \eqref{bg_c1alpha}) arbitrarily large.
\end{proof}

 We now want to control the error
\[
\ttp \coloneqq \tp -\otp.
\]
To this end, we first let $\gamma \in (0,1)$ be sufficiently  small so that
\eqnb\label{choice_gamma}
3+2\gamma -2s <0,
\eqne
we fix $\eta\in (0,1)$ to be sufficiently small so that
\eqnb\label{choice_eta}
3+2\gamma -2s+2\eta <0,
\eqne
and we fix $\epsilon \in (0,1)$ be sufficiently small so that
\eqnb\label{eps_cond}
2-2s + \epsilon + \eta <0.
\eqne
We relate $\lt$ with $\nt$ via
\eqnb\label{lt_vs_nt}
\nt= \lt^{1-\eta },
\eqne
and we relate $\lt$ to $\lambda$ via 
\eqnb\label{lt_vs_lambda}
\lt = \lambda^B,
\eqne
where $B>1$ is a large constant, to be fixed below.
\begin{lemma}[Error of the perturbation]\label{L_pert}
Given $K,P \geq 1 $ there exists $c>0$ such that, for all sufficiently large $\lambda\geq 1$ (depending on $K,P$),
\eqnb\label{pert_ests}
\begin{split}
\| \ttp \|_{L^\infty } &\lec \lambda^c \lt^{2-\alpha -2s + \eta (s-1) },\\
\| \ttp \|_{C^1} &\lec \lambda^c  \lt^{3-2s + \eta(s-1)},\\
\| \ttp \|_{C^{1,a}}  &\lec \lambda^c  \lt^{3+2a-2s+2\eta },\\
\| \ttp \|_{H^s} &\lec  \lambda^c \lt^{-\eta },\\
\| \ttp \|_{H^4} &\lec  \lt^c 
\end{split}
\eqne
for all $a\in (0,1)$, $t\in [0,T]$.
\end{lemma}
Before we prove the lemma, we fix $B$ to be sufficiently large so that the quantities on the right-hand side of \eqref{pert_ests} with $a\coloneqq \gamma$ are decreasing with $\lambda \to \infty$. We also note that (by \eqref{def_otp}), 
\[ \supp\, \otp (t)= \bigcup_{m=0}^{P-1} R_m (D(t)  ),
\]
where
\[
D(t)\coloneqq  \widetilde{ \eta }( \supp \psi_0 (\cdot ,0) ,t)
\]
 (recall \eqref{teta_def} that $\widetilde{\eta}$ denotes Lagrangian trajectory of $\tv$), so that in particular
\eqnb\label{supp_otp}
\mathrm{diam}\, D(t) \lec \lt^{-1}N^{ct}
\eqne
\eqnb\label{supp_otpold}
\mathrm{diam}\, (\supp\, \otp (t) ) \lec \lt^{-1}N^{ct}
\eqne
and
\eqnb\label{otp_holder_norms}
\| \otp \|_{C^{k,\alpha} } \lec \lt^{k+1+\alpha -s} \nt^{k+\alpha -s}  N^{c(k+\alpha )t}
\eqne
for $k\geq 0$, $\alpha \in [0,1)$. Thus also
\eqnb\label{votp_holder_norms}
\| v[\otp ] \|_{C^{k,\alpha} } \lec \lt^{k+1+\alpha -s} \nt^{k+\alpha -s}  N^{c(k+\alpha )t} \log (\lt \nt ).
\eqne

Moreover, letting 
\eqnb\label{ov_otp_def}
\ov [\otp ] \coloneqq  \lt^{1-s} \nt^{-s}K^{-1} \sum_{m=0}^{P-1} g \circ R_{-m} (x' ) \begin{pmatrix}
0\\
\cos (\lt \nt x'_1 )
\end{pmatrix},
\eqne
where $x'\coloneqq \ee^{-A(t)}x-b(t) \sim N^t x -b(t)$, we have that, as in \eqref{G_ests},
\[
\ov [\otp ] \cdot \na \otp = \lt^{3-2s} \nt^{-2s} K^{-2} g(\lt x' ) \p_2 f (\lt x' ) \cos (\lt \nt x'_1 ) \sin (\lt \nt x'_1 ),
\]
and so
\[
\| \ov [\otp ] \cdot \na \otp \|_{C^{k,\alpha }} \lec  \lt^{k+\alpha +3-2s} \nt^{k+\alpha -2s} N^{c(k+\alpha )t}.
\]
By Corollary~\ref{cor_vel_err} we have $\| v[\otp ] - \ov  [\otp ] \|_{C^{k,\alpha }} \lec_\epsilon  \lt^{1+k+\alpha -s } \nt^{-1+k+\alpha-s +\epsilon } N^{c(2k+3+\alpha + \epsilon )t}$ for every $\epsilon \in (0,1)$, and so
\eqnb\label{333}
\| v [\otp ] \cdot \na \otp \|_{C^{k,\alpha }} \lec_\epsilon K^{-2} \lt^{k+\alpha +3-2s} \nt^{k+\alpha +\epsilon -2s} N^{c(2k +4+\alpha + \epsilon ) t}
\eqne
for $k\geq 0$, $\alpha \in [0,1)$, $\epsilon \in (0,1)$.

Moreover, by \eqref{c2_error} and \eqref{supp_otp},
\eqnb\label{vtilde_err_C1gamma}
\| \tv - v [\tb ] \|_{C^{1,a}(\supp\, \otp)} \leq \| v [\tb ] \|_{C^{2 } }  \left( \mathrm{diam}\, D(t) \right)^{1-a } \lec  \lambda N \log N \cdot  \lt^{a -1 } N^{c(1-a  )t}
\eqne
 and 
\eqnb\label{vtilde_err_Linfty}
\| \tv - v [\tb ] \|_{L^\infty (\supp\, \otp)} \leq \| v [\tb ] \|_{C^{2 } }  \left( \mathrm{diam}\, D(t) \right)^{2 } \lec_{K,P}  \lambda N \log N \cdot  \lt^{-2 } N^{2c t}.
\eqne
As for the Sobolev norms of $\otp$ we have
\[
\begin{split}
\| \otp \|_{L^2} &\sim P^{1/2} \lt^{-s} \nt^{-s} ,\\
\| \otp \|_{\dot H^1} &\sim P^{1/2}\lt^{1-s} \nt^{1-s}\ee^{-A(t)} \sim \lt^{1-s} \nt^{1-s } N^t  ,\\
\| \otp \|_{\dot H^2} &\sim P^{1/2}\lt^{2-s} \nt^{2-s} N^{2t},
\end{split}
\]
which imply that (since $\| f \|_{\dot H^1}\lec \| f \|_{L^2}^{(\beta -1)/\beta } \| f \|_{\dot H^\beta  }^{1/\beta }$ for $\beta >1$)
\[
\| \otp \|_{\dot H^\beta} \gec  \| \otp \|_{\dot H^1}^\beta \| \otp \|_{L^2}^{1-\beta}  \sim_{P} K^{-1}   \lt^{\beta (1-s) +(1-\beta )(-s)}\nt^{\beta (1-s)+(1-\beta ) (-s)}  N^{\beta t}= K^{-1}\lt^{\beta -s }\nt^{\beta -s }  N^{\beta t}
\]
and 
\[
\| \otp \|_{\dot H^\beta } \lec \| \otp \|_{\dot H^1}^{2-\beta } \| \otp \|_{\dot H^2}^{\beta -1 } \sim_{P} K^{-1}  \lt^{\beta -s }\nt^{\beta -s }  N^{\beta t}
\]
for $\beta >1$. This, and the fourth inequality in \eqref{pert_ests} show that
\[
\| \tp \|_{\dot H^\beta } \sim \lambda^{B(2-\eta )(\beta -s) +\frac{2-s}{s-1}\beta t } (\log \lambda )^{-\frac{\beta t }{s-1}} K^{-1-\frac{\beta t}{s-1} }
\]
for $\beta \in (1,2)$, where we also used the relationship \eqref{lambdaN}, $\lambda^{2-s}= N^{s-1} K  \log N$, to deduce that 
\[
(2-s) \log \lambda \leq \log N \leq \frac{2-s}{s-1} \log \lambda
\]
for sufficiently large $\lambda$.  We can now prove Lemma~\ref{L_pert}.

\begin{proof}[Proof of Lemma~\ref{L_pert}.]
As in the proof of Lemma~\ref{L_bg} we first assume that $\| \ttp \|_{C^{1,\gamma }} \lec 1$ on $[0,T]$, and the resulting estimates will follow from below computations using a continuity argument.

 We see that $\ttp$ satisfies the PDE
\eqnb\label{eq_ttp}\begin{split}
\p_t \ttp  &= - v[\tb + \tp  ] \cdot \nabla \tp - v[\tp ] \cdot \nabla \tb + \tv  \cdot \nabla \otp  \\
& =-v[\tb ] \cdot \nabla \ttp - v[\tp ]\cdot \nabla \tp - v[\tp ] \cdot \nabla \tb -  \left(  v[\tb  ] - \tv  \right) \cdot \nabla \otp  \\
&= -v[\tb ] \cdot \nabla \ttp - v[\ttp+ \otp  ]\cdot \nabla (\ttp + \otp ) - v[\ttp ] \cdot \nabla \tb -  \left(  v[\tb  ] - \tv  \right) \cdot \nabla \otp - v[\otp ]\cdot \na \tb \\
&= -v[\tb ] \cdot \nabla \ttp - v[\ttp ] \cdot \left( \na \ttp + \na \otp +\na \tb \right) - v[\otp ] \cdot \left( \na \ttp + \na \otp \right) -  \left(  v[\tb  ] - \tv  \right) \cdot \nabla \otp - v[\otp ]\cdot \na \tb \\
&= -v[\tb ] \cdot \nabla \ttp - v[\ttp ] \cdot \na \ttp - v[\otp ] \cdot \na \ttp - v[ \ttp ] \cdot \na \otp - v[ \otp  ]\cdot \nabla  \otp  - v[\ttp ] \cdot \nabla \tb -  \left(  v[\tb  ] - \tv  \right) \cdot \nabla \otp  - v[\otp ]\cdot \na \tb
\end{split}
\eqne
with homogeneous initial condition. We have 
\[
\begin{split}
\frac{\d }{\d t} \| \ttp \|_{L^\infty } 
&\lec 
\| v [\ttp ] \|_{L^\infty } \| \nabla \otp + \nabla \tb \|_{L^\infty } 
+
\| v [\otp ] \cdot  \nabla \otp  \|_{L^\infty } 
+ 
\| v[\tb ] - \widetilde{v} \|_{L^\infty (\supp\, \otp )} \| \nabla \otp \|_{L^\infty }+ \| v [\otp ] \|_{L^\infty } \| \tb \|_{C^1} \\
&\lec 
\| \ttp \|_{L^\infty } \left( \lt^{2-s} \nt^{1-s} N^{t} +  \log N \right) 
\\
&\hspace{1cm}+ C_{K,P}\lambda^c  \left(
\lt^{3-2s} \nt^{\epsilon-2s} N^{(5 + \epsilon )t} 
+ 
   \lt^{-2 }\cdot \lt^{2-s}\nt^{1-s} N^t+  \lt^{1-s} \nt^{-s} \log (\lt \nt ) \right) \\
&\lec 
\| \ttp \|_{L^\infty } \left( \lambda^c \lt^{3-2s+\eta(s-1)}  +  \log N \right) 
+ C_{K,P}
\lambda^c \lt^{3-4s+2s\eta+\epsilon (1-\eta )}  
 + 
  \lambda^c \lt^{1 -2s + \eta (s-1) } \\
&\lec 
\| \ttp \|_{L^\infty }    \log N + C_{K,P}
\lambda^c \lt^{1 -2s + \eta (s-1) }
\end{split}
\]
for all $t\in [0,T]$, where we used \eqref{otp_holder_norms}, \eqref{otb_Cka_est} \eqref{ThetaC1}, \eqref{333} and \eqref{vtilde_err_Linfty} in the second line and we recalled \eqref{lt_vs_nt} in the third line.
Thus, by the ODE fact \eqref{ode_fact}
\eqnb\label{ttp_Linfty}
\| \ttp \|_{L^\infty } \lec_{K,P} \lambda^c \lt^{1 -2s + \eta (s-1) }
\eqne
and a similar estimate holds for $\| v [\ttp ]\|_{L^\infty}$. Furthermore, for $a\in [0,1)$,
\[
\begin{split}
\frac{\d }{\d t}& \| \ttp \|_{C^{1,a } }
\lec 
\| v[\tb ] \|_{C^1 } \| \ttp \|_{C^{1,a}} + \| v[\tb ] \|_{C^{1,a} } \| \ttp \|_{C^1} \\
&\hspace{0.5cm}+
\| v [\ttp ] \|_{C^1 } \| \ttp \|_{C^{1,a}}+\| v [\ttp ] \|_{L^\infty } \| \nabla \otp + \nabla \tb \|_{C^{1,a} }  +  \| v [\ttp ] \|_{C^{1,a} } \left( \| \ttp \|_{C^1} +\| \otp +  \tb \|_{C^{1} } \right) \\
&\hspace{0.5cm}+
\| v[\otp ] \|_{C^1} \| \ttp \|_{C^{1,a}} + \| v[\otp ] \|_{C^{1,a} } \| \ttp \|_{C^{1}}  \\
&\hspace{0.5cm}+
\| v [\otp ] \cdot  \nabla \otp  \|_{C^{1,a} } \\
&\hspace{0.5cm}+ 
\| v[\tb ] - \widetilde{v} \|_{L^\infty (\supp\, \otp )} \|  \otp \|_{C^{2,a} }+ \| v[\tb ] - \widetilde{v} \|_{C^{1,a} (\supp\, \otp )} \|  \otp \|_{C^{1} } \\
&\lec \| v[\tb ] \|_{C^1 } \| \ttp \|_{C^{1,a}} +\| v[\tb ] \|_{C^{1,a} }  \| \ttp \|_{C^1} \\
&\hspace{0.5cm}+ 
\| v [\ttp ] \|_{C^1 }  \| \ttp \|_{C^{1,a}} +C_{K,P} \lambda^c \lt^{1 -2s + \eta (s-1)} \cdot  \lt^{3+a-s} \nt^{2+a-s} + \| v[ \ttp ] \|_{C^{1,a}} \left( \| \ttp \|_{C^1}+  \lambda^c \lt^{2-s} \nt^{1-s}\right) \\
&\hspace{0.5cm}+
 \lambda^c  \lt^{2-s}\nt^{1-s} \log (\lt \nt) \| \ttp \|_{C^{1,a}} + \lambda^c \lt^{2+a-s} \nt^{1+a-s}\log (\lt \nt ) \| \ttp \|_{C^1} \\
&\hspace{0.5cm}+ C_{K,P} \lambda^c \left(  \lt^{4+a-2s} \nt^{1+a+\epsilon -2s} +\lt^{-2} \cdot \lt^{3+a-s} \nt^{2+a-s} +  \lt^{a-1} \cdot \lt^{2-s} \nt^{1-s} \right) \\
&\lec \left( \| v[\tb ] \|_{C^1 } + \| v[\ttp ]\|_{C^1} \right) \| \ttp \|_{C^{1,a}} +\left( \| v[\tb ] \|_{C^{1,a} } +\| v[\ttp ] \|_{C^{1,a} } + \lambda^c \lt^{2+a-s} \nt^{1+a-s} \log (\lt \nt ) \right) \| \ttp \|_{C^1} \\
&\hspace{0.5cm}+ \| \ttp \|_{C^{1,a}}  \lambda^c \lt^{3-2s + \eta(s-1)} \\
&\hspace{0.5cm}+ C_{K,P}\lambda^c \left( \lt^{6+2a -4s + \eta(2s-3-a)} + \lt^{3+2a-2s + \eta(s-1-a)} + \lt^{5+2a-4s +\epsilon + \eta(2s-1-a-\epsilon )} \right.\\
&\hspace{0.5cm}\left.+ \lt^{3+2a-2s + \eta(s-2-a)}+ \lt^{2+a-2s + \eta(s-1)} \right)\\
&\lec \left( \| v[\tb ] \|_{C^1 } + \| v[\ttp ]\|_{C^1} \right) \| \ttp \|_{C^{1,a}} +\left( \| v[\tb ] \|_{C^{1,a} } +\| v[\ttp ] \|_{C^{1,a} } + \lambda^c \lt^{3+2a-2s+\eta(s-1-a)}\log \lt \right) \| \ttp \|_{C^1} \\
&\hspace{0.5cm}+ \| \ttp \|_{C^{1,a}}  \lambda^c \lt^{3-2s + \eta(s-1)} +C_{K,P} \lambda^c  \lt^{3+2a-2s + \eta(s-1-a)} ,
\end{split}
\]
where we used the fact that $N^{ct} \lec \lambda^c $ (recall $t\leq T$), \eqref{ttp_Linfty}, \eqref{otp_holder_norms}, \eqref{votp_holder_norms}, \eqref{333} in the second inequality, as well as \eqref{eps_cond} in the form $5+2a-4s+\epsilon + \eta < 3+2a-2s$ in the last inequality. Thus, for 
$a=0$ we obtain 
\[
\frac{\d }{\d t } \| \ttp \|_{C^1} \lec \| \ttp \|_{C^1} \log N + C_{K,P}\lambda^c  \lt^{3-2s + \eta(s-1)}
\]
(as long as $\| v [\ttp ]\|_{C^{1,\gamma }}$ remains bounded). Thus the ODE fact~\eqref{ode_fact} implies that
\eqnb\label{ttp_C1}
\| \ttp \|_{C^1} \lec_{K,P} \lambda^c  \lt^{3-2s + \eta(s-1)}
\eqne
for $t\in [0,T]$. On the other hand, for $a\in (0,1)$, we get
\[
\begin{split}
\frac{\d }{\d t } \| \ttp \|_{C^{1,a}} &\lec  \| \ttp \|_{C^{1,a}}\left( \log N+ \lambda^c \lt^{3-2s+ \eta (s-1)} \right) \\
&+C_{K,P} \lambda^c \left( \lt^{3-2s+\eta (s-1)} + \lt^{6+2a-4s + \eta (2s -2-a)} \log \lt + \lt^{3+2a-2s + \eta (s-1-a)} \right) \\
& \lec \| \ttp \|_{C^{1,a}} \lambda^c \lt^{3-2s+2 \eta } +C_{K,P} \lambda^c  \lt^{3+2a-2s+2\eta },
\end{split}
\]
and so, by \eqref{choice_eta},
\[
\| \ttp \|_{C^{1,a}}  \lec_{K,P}\lambda^c  \lt^{3+2a-2s+2\eta }
\]
for $t\in [0,T]$, as required, where in the above computation the implicit constant in ``$\lec$'' does  not depend on $K,P$. 

By \eqref{choice_gamma}--\eqref{choice_eta} we thus have the $C^{1,\gamma}$ norm under control, and so we can now estimate the $H^s$ error.
\[
\begin{split}
\frac{\d }{\d t} \| \ttp \|_{H^k} &\lec \| v[\tb ] \|_{C^1 } \| \ttp \|_{H^{k}} + \| v[\tb ] \|_{H^k } \| \ttp \|_{C^1} \\
&+
\| v [\ttp ] \|_{C^1 }\| \ttp \|_{H^k}+\| v [\ttp ] \|_{L^\infty } \|  \otp + \tb \|_{H^{k+1} }  +  \| v [\ttp ] \|_{H^k } \left( \| \ttp \|_{C^1} +\| \otp +  \tb \|_{C^{1} } \right) \\
&+
\| v[\otp ] \|_{C^1 } \| \ttp \|_{H^k} + \| v[\otp ] \|_{H^k  } \| \ttp \|_{C^{1}}  \\
&+
\| v [\otp ] \cdot  \nabla \otp  \|_{H^k  } \\
&+ 
\| v[\tb ] - \widetilde{v} \|_{L^\infty (\supp\, \otp )} \|  \otp \|_{H^{k+1} }+ \| v[\tb ] - \widetilde{v} \|_{H^k (\supp\, \otp )} \|  \otp \|_{C^{1} } \\
&\hspace{-1cm}\lec  \| \ttp \|_{H^k}\log N + C_{K,P} \left( \lambda^c \lt^{3-2s + \eta(s-1)}\right.\\
&+\lambda^c \lt^{3-2s+\eta(s-1)} \| \ttp \|_{H^k}+ \lambda^c \lt^{1 -2s + \eta (s-1)}  \cdot \lt^{k+1-s}\nt^{k+1-s}  + \| \ttp \|_{H^k} \lambda^c \left( \lt^{3-2s+\eta(s-1)} + \lt^{2-s}\nt^{1-s} \right)\\
&+\lt^{2-s}\nt^{1-s} \log (\lt \nt )\| \ttp \|_{H^k} + \lt^{k-s}\nt^{k-s} \cdot \lambda^c \lt^{3-2s+\eta (s-1)}\\
&+ \lambda^c \lt^{k+2-2s}\nt^{k+\varepsilon-2s}\\
&+\left.\lambda^c \lt^{-2} \cdot (\lt \nt )^{k+1-s} + \lambda^c \lt^{-1} \cdot \lt^{2-s}\nt^{1-s}\right)\\
&\hspace{-1cm}\leq \| \ttp \|_{H^k} \log N + C_{K,P}\lambda^c \left( \lt^{3-2s + \eta (s-1) } + \lt^{k+2-3s+\eta(s-1) }\nt^{k+1-s} \right.\\
&\hspace{5cm}\left.+\lt^{k+3-3s+\eta(s-1) }\nt^{k-s} + \lt^{k+2-2s} \nt^{k-2s} + \lt^{k-1-s} \nt^{k+1-s}  \right) \\
&\hspace{-1cm}=\| \ttp \|_{H^k} \log N + C_{K,P}\lambda^c \left( \lt^{3-2s + \eta (s-1) } + \lt^{2k+3-4s+\eta(2s-2-k) } \right.\\
&\hspace{5cm}\left.+\lt^{2k+3-4s+\eta(2s-1-k) } + \lt^{2k+2-4s + \eta(2s-k)} +\lt^{2k-2s+\eta (s-1-k)}  \right) \\ 
& \hspace{-1cm}\leq\| \ttp \|_{H^k} \log N +C_{K,P}\lambda^c \lt^{2k-2s+\eta (s-1-k)}  ,
\end{split}
\]
where we used \eqref{333}, \eqref{ttp_C1}. In particular
\[
\| \ttp \|_{H^2} \lec_{K,P} \lambda^c \lt^{4-2s+\eta(s-3) } \qquad \text{ and }\qquad  \| \ttp \|_{H^1} \lec_{K,P} \lambda^c \lt^{2-2s+\eta(s-2) }.
\]
Thus, by interpolation
\eqnb\label{byinter}
\| \ttp \|_{H^s} \lec \| \ttp \|_{H^1}^{2-s} \| \ttp \|_{H^2}^{s-1} \lec_{K,P} \lambda^c \lt^{-\eta},
\eqne
as required. Similarly, we obtain that $\| \ttp \|_{H^4} \lec_{K,P} \lt^c$ for some $c>0$.
\end{proof}

We can now conclude the proof of Theorem~\ref{T00}. We first let $\varepsilon>0$ and $T\in (0,1)$ be given by Lemma~\ref{L_bg} and we take $\lambda$ large enough so that the estimates in Lemmas~\ref{L_bg} and \ref{L_pert} are valid on $[0,T]$. The estimates \eqref{black_box_claims} and \eqref{black_box_Hb} then follow from the control of the background and perturbation errors in Lemmas~\ref{L_bg} and \ref{L_pert} and the size of the pseudosolutions $\otb$, $\otp$ in the corresponding norms. 

\section{Continuous loss of regularity}\label{sec_pf_main}
Here we prove the main result, Theorem~\ref{T01}.

\subsection{Existence and uniqueness lemmas}

Here we discuss two lemmas that will help us with the gluing process before we can proceed to the final construction.

The first lemma is concerned with a given solution $\theta$ to the SQG equation \eqref{sqg} that is supported away from the origin, and shows that one can glue $\theta_{K,\lambda}$ to it, provided $\lambda>0$ is sufficiently large, where $\theta_{K,\lambda}$ is the norm inflation solution provided by Theorem~\ref{T00}.

\begin{lem}[Existence lemma]\label{existence}
Let $s\in(\frac{3}{2},2)$, $P\in \N$, $K>1$, and let $\lambda>0$ be sufficiently large. Let $\theta_{\lambda}$ denote the solution to the SQG equation on $[0,T_0]$ given by Theorem~\ref{T00} (where $T_0>0$ is given by Theorem~\ref{T00}), and let $\theta$ be another solution to the SQG equation \eqref{sqg} such that $\theta$ is odd-odd symmetric, $P$-fold symmetric, smooth for $(x,t)\in \R^2 \times [0,T_0]$  with compact support $\supp(\theta(\cdot ,t))\subset B_{R}\setminus B_{r}$ for some $R>r>0$.\\

    If $P$ is sufficiently large (depending on $s$ only) and $\lambda\geq 1$ is sufficiently large (depending on $K$, $P$,$s$, $\theta$ and $T_{0}$) then the unique local-in-time smooth solution $\theta_{\rm new}$ to SQG \eqref{sqg} with initial conditions $\theta(x,0)+\theta_{\lambda}(x,0)$ exists until $T_0$, and
    $$\supp\, \theta_{\rm new} \subset  B_{R+\lambda^{-\frac12}}\setminus B_{\lambda^{-1/2}/4} \qquad \text{ and }\qquad  \|\theta+\theta_{\lambda}-\theta_{\rm new}\|_{H^3}\leq \lambda^{-1}$$
    for all $t\in [0,T_0]$.
\end{lem}
\begin{proof}
We will keep the dependence on the values of $K,s$ and $T_{0}$ implicit, as well as the dependence on $\theta$, keeping only the dependence on $P$ explicit (so all the constants involved should be $C_{K,s,T_{0},\theta}$, but we will omit the subindices for simplicity).

We first note that, for small times $\theta_{\rm new}$ can be decomposed into the sum of two disjointly supported pieces, one approximating $\theta$ and the other $\theta_{K,\lambda}$.  Namely, by a continuity argument, we have that
$$\theta_{\rm new}=\widetilde{\theta}+\widetilde{\theta}_{\lambda}$$
for some $t>0$. Let $T\leq T_0$ be the largest time such that the above is true, together with quantitative estimates
 \eqnb\label{theta_003}
 \|\theta -\widetilde{\theta}\|_{H^3},\|\theta_{\lambda}-\tilde{\theta}_{\lambda}\|_{H^3}< \frac{\lambda^{-1}}{2}
 \eqne
    and
    \eqnb\label{theta_suppa}
    \begin{split}
\frac{r}2  \leq \left| \widetilde{\phi }(x,t) \right|  &\leq R+\lambda^{-1/2}   \hspace{2cm}\text{ for } x\in \supp \theta (\cdot ,0) \qquad \text{ and }\\
 \lambda^{-1/2}/4   \leq \left| \widetilde{\phi }(x,t) \right| & \leq 4\lambda^{-1/2}  \hspace{2.6cm} \text{ for }\lambda^{-1/2}/2\leq |x|\leq 2 \lambda^{-1/2}
 \end{split}
    \eqne
for all $t\in [0,T]$, where $\widetilde{\phi}$ denotes the Legrangian trajectory of $v[\theta_{\rm new}]$, i.e. the solution to
\[\p_{t}\widetilde{\phi}=v[ \theta_{\rm new} ]\circ \widetilde{\phi},\qquad \widetilde{\phi}(x,0)=x.\]
Note that $T>0$ by continuity, and that \eqref{theta_supps} implies that 
\eqnb\label{theta_supps}\text{supp}\,\widetilde{\theta}\subset B_{R+\lambda^{-\frac12}}\setminus B_{{r}/{2}},\qquad  \text{supp}\,\widetilde{\theta}_{\lambda}\subset B_{4\lambda^{-\frac{1}{2}}}\setminus B_{\lambda^{-\frac{1}{2}}/4}
    \eqne
for all $t\in [0,T]$, due to the transport structure of the SQG equation.\\

If $T= T_0$ then the claim of the lemma follows. Thus, let us suppose that $T<T_0$. We show below that, for $t\in [0,T]$, properties \eqref{theta_003}--\eqref{theta_suppa} hold with, repectively, strictly smaller right-hand side and strictly smaller restrictions on $|\widetilde{\phi}|$.   This contradicts the definition of $T$, and so finishes the proof. \\

For every $t\in [0,T]$ both $\widetilde{\theta}$ and $\widetilde{\theta}_{\lambda}$ are advected by $v[\theta_{\rm new}]=v[\widetilde{\theta}+\widetilde{\theta}_{\lambda}]$, i.e. they satisfy the equations
   \[
    \begin{split} \p_{t}\widetilde{\theta}+v[\widetilde{\theta}+\widetilde{\theta}_{\lambda}] \cdot\nabla \widetilde{\theta}&=0,\qquad \widetilde{\theta}(x,0)=\theta(x,0),\\
    \p_{t}\widetilde{\theta}_{\lambda}+v[\widetilde{\theta}+\widetilde{\theta}_{\lambda}]\cdot\nabla \widetilde{\theta}_{\lambda}&=0,\qquad \widetilde{\theta}_{\lambda}(x,0)=\theta_{\lambda}(x,0).
    \end{split}
    \]    

From the evolution equation for $\Theta\coloneqq \theta-\tilde{\theta}$ we get
    \begin{align*}
        \p_{t}\|\Theta\|_{H^3}\leq C\|\Theta\|^2_{H^3}+\|\Theta\|_{H^3}\|\theta\|_{H^4}+\|\Theta\|_{H^3}\|v[\tilde{\theta}_{\lambda}]\|_{C^{4}(\R^2\setminus B_{\frac{r}{2}})(0)}+\|\theta\|_{H^4}\|v[\tilde{\theta}_{\lambda}]\|_{C^{3}(\R^2\setminus B_{\frac{r}{2}})(0)}.
    \end{align*}
  Taking $\lambda$ big enough so that $\supp\, \widetilde{\theta}_{\lambda}\subset B_{{r}/{4}}$,
\eqnb\label{temp_003}\|v[\tilde{\theta}_{\lambda}]\|_{C^4(B_{{r}/{2}}^c)}\leq C\|\tilde{\theta}_{\lambda}\|_{L^1}\leq C_{P}\lambda^{-2}.
\eqne
   Thus, since $\|\Theta\|_{H^3}<\frac{\lambda^{-1}}{2} \leq 1$ for $t\in [0,T)$ and $\| \theta \|_{H^4}\leq C$, we obtain
    \begin{align*}
        \p_{t}\|\Theta\|_{H^3}\leq C\|\Theta\|_{H^3}+C_{P}\lambda^{-2}.
    \end{align*}
    Hence, using the ODE fact \eqref{ode_fact},
    \eqnb\label{theta001}
    \|\Theta\|_{H^3}\leq C_{P}\lambda^{-2}\leq \frac{\lambda^{-1}}{2}
    \eqne
    for $t\in [0,T]$, if $\lambda$ is sufficiently large.
    
    As for $\Theta_{\lambda}\coloneqq \theta_{\lambda } - \tt_{\lambda} $, we start by obtaining bounds in $L^2$, where we get
    \eqnb\label{theta_002}
    \frac{\d}{\d t} \|\Theta_{\lambda}\|_{L^2}\leq \|\Theta_{\lambda}\|_{L^2}\|\tilde{\theta}_{\lambda}\|_{C^1}+\|v[\tilde{\theta}]\cdot\nabla \tilde{\theta}_{\lambda}\|_{L^2}.
    \eqne
    
 Since $\widetilde{\theta}$ and $\widetilde{\theta}_{\lambda}$ are $P$-fold symmetric (as both initial conditions $\theta(\cdot ,0)$ and $\theta_{\lambda}(\cdot ,0)$ are), Lemma~\ref{decay} gives that
    $$D^J v[\widetilde{\theta}](0,t)=0 \qquad \text{ if }J\leq P-1.$$

    Furthermore, recalling \eqref{theta_supps},
    $$\|v[\tt]\|_{C^{J}(B_{4\lambda^{-{1}/{2}}})}\lec_{J} \|\tt \|_{L^1}\lec 1.$$
    This, and the Taylor expansions at $x=0$ of $v[\widetilde{\theta}]$ and its fourth order partial derivatives give
  \eqnb\label{theta_007}\|v[\tt]\|_{C^{4}(B_{4\lambda^{-{1}/{2}}})}\leq C \lambda^{\frac{-P+5}{2}}
  \eqne
    for $t\in [0,T)$. Applying this in \eqref{theta_002}  and using the bounds for the $C^1$ norm of $\widetilde{\theta}_{\lambda}$ (recall~\eqref{black_box_claims}), which in particular, due to the smallness of the support, give bounds in $H^1$, gives
    $$\frac{\d}{\d t} \|\Theta_{\lambda}\|_{L^2}\leq C\log \lambda  \|\Theta_{\lambda}\|_{L^2}+C_{P}\lambda^{\frac{-P+5}{2}} \log \lambda ,$$
    so that (recall the ODE fact \eqref{ode_fact})
    \eqnb\label{hereisP}
    \|\Theta_{\lambda}\|_{L^2}\leq C_{P}\ee^{Ct\log \lambda } \lambda^{\frac{-P+5}{2}}\log \lambda \lec C_{P}\lambda^{\frac{-P+C}{2}}.
    \eqne
 Interpolation between this and the assumed $H^3$ bound \eqref{theta_003} gives 
    $$\|\Theta_{\lambda}\|_{H^{\beta}}\leq C_{P}\lambda^{\frac{(3-\beta)}{3}\frac{(-P+C)}{2}}.$$
    for  $\beta\in[0,3]$, $t\in [0,T]$.
    This lets us improve the $H^3$ bound, since 
    \begin{align*}
        \frac{\d }{\d t}\|\Theta_{\lambda}\|_{H^3}&\lec \|\Theta_{\lambda}\|^2_{H^3}+\sum_{j=0}^{3}\|\Theta_{\lambda}\|_{H^j}\|\theta_{\lambda}\|_{C^{4-j}}\\
        &+\sum_{j=0}^{3}\|\Theta_{\lambda}\|_{H^j}\left(\|v[\theta_{\lambda}]\|_{C^{4-j}}+\|v[ \tilde{\theta}]\|_{C^{4-j}(B_{4\lambda^{-\frac{1}{2}}}(0))})+\|v[\tilde{\theta}]\cdot\nabla \theta_{\lambda} \|_{H^{3}}\right)\\
        &\lec  (1+ \log \lambda )\|\Theta_{\lambda}\|_{H^3}+C_{P}\lambda^{\frac{-P+C}{6}},
    \end{align*}
    so that 
     \[
     \|\Theta_{\lambda}\|_{H^3}\leq  C_{P}\lambda^{\frac{-P+C}{6}}
     \]
     for $t\in [0,T]$. Hence, we can take $P$  big enough so that
    \eqnb\label{theta_006}
    \|\Theta_{\lambda}\|_{H^3}\leq C_{P}\lambda^{-2} \leq \frac{\lambda^{-1}}{2}
    \eqne
for sufficiently large $\lambda$.

As for  \eqref{theta_suppa}, we let $\phi$ denote the Lagrangian trajectory of $v[\theta ]$ and we note that 
 \[
 \begin{split}
 \p_{t}|\phi-\widetilde{\phi}|&\leq |\phi-\widetilde{\phi}|\|v[\theta]\|_{C^1(B_{r/2}^c)} + \|v[\theta-\theta_{\rm new}]\|_{L^\infty (B_{r/2}^c)}\\
 &\lec_\theta |\phi-\widetilde{\phi}| + C_{P}\lambda^{-2},\\
 \end{split}
 \]
 for each $x\in B_r^c$ and $t\in [0,T]$, where we used \eqref{theta001} in the form $\|v[\Theta ] \|_{L^\infty } \lec \|v[\Theta ] \|_{H^3 } \lec \| \Theta \|_{H^3} \lec  C_P \lambda^{-2} $ and  \eqref{temp_003}.
Hence, for large $\lambda$ 
\eqnb\label{theta_004}
|\phi(x,t)-\widetilde{\phi}(x,t)| \leq \lambda^{-1} 
\eqne
for $x\in \supp\, \theta (\cdot , 0)$, $t\in [0,T]$,  and so, since $|\phi(x,t)|\geq r$ for all such $x,t$ (by the assumptions on $\theta$), the first line of \eqref{theta_suppa} follows.

 For the second line of \eqref{theta_suppa}, we first note that
\[
\p_t \widetilde{\phi } = v[ \theta_{\rm new} ] (\widetilde{\phi }) = v[ \theta ] (\widetilde{\phi }) +  v[ \theta_\lambda ] (\widetilde{\phi }) +  v[ \widetilde{\theta} + \widetilde{\theta_\lambda} - (\theta + \theta_\lambda) ] (\widetilde{\phi }).
\]
Thus we can use the decay Lemma~\ref{decay} and \eqref{theta_003} to get 
 \[
 \begin{split}
 \p_{t}|\widetilde{\phi}|&\lec \|v[{\theta}]\|_{L^{\infty}(B_{4\lambda^{-{1}/{2}}})} + \| v[ \theta_\lambda ] \|_{L^\infty (B_{4\lambda^{-{1}/{2}}}) } + \| \widetilde{\theta} - \theta \|_{H^3} +\| \widetilde{\theta_\lambda} -  \theta_\lambda\|_{H^3} \\
 &\lec  C_{P,\theta} \lambda^{-\frac{P-1}2} +\|\theta_\lambda \|_{L^{\infty}}(2+\log \|\theta_\lambda \|_{C^1})+\lambda^{-1}\leq C_{P,\theta }\lambda^{-1}\log \lambda 
 \end{split}
 \]
 for each $x\in B_{2\lambda^{-1/2}}$, $t\in [0,T)$, where we also  recalled the $L^\infty$ and $C^1$ bounds on $\theta_\lambda$ (from \eqref{black_box_claims}) to write
 \[
 \|v[\theta_{\lambda}]\|_{L^{\infty}}\lec\| \theta_{\lambda}\|_{L^{\infty}}\log(2+\|\theta_{\lambda}\|_{C^1}) \lec \lambda^{-1} \log \lambda
 \]

 Hence
\[
| \widetilde{\phi } (x,t) - x| \lec_{C,\theta,T_0} \lambda^{-1}\log \lambda ,
\] 
and so, taking sufficiently large $\lambda$ we obtain the second line of \eqref{theta_suppa}. 
\end{proof}
We note that the above proof uses the choice of large $P$ only to control  $\Theta_{\lambda}$ and the support of $\widetilde{\theta}_{\lambda}$. For example, the proofs of \eqref{theta001} and \eqref{theta_004} do not require large $P$, and, from the assumed properties \eqref{theta_003}--\eqref{theta_supps},  they only use  that 
\eqnb\label{theta_01}
\supp\,\widetilde{\theta}_{\lambda}\subset B_{r/4}
\eqne
(recall that we used \eqref{temp_003}). Similarly, taking $P$ large enabled us (thanks to \eqref{theta_007}) to ensure that $\widetilde{\theta}_{\lambda}$ remained localized similarly to $\theta_{\lambda}$ (recall~\eqref{theta_supps}). This shows that, if we were able to guarantee \eqref{theta_01} using some other tool or assumption, then we would be able not only to control the support of the outer part $\widetilde{\theta}$ of the new solution $\theta_{\rm new}$ (i.e. the first claim of \eqref{theta_supps}), but also to ensure that 
\[
\| \theta_{\rm new} - \theta \|_{H^3(B_{r/2}^c)} < \lambda^{-1},
\]
as in \eqref{theta_003}. One way to guarantee \eqref{theta_01} is to impose an additional control of $\theta_{\rm new}$, for example that $\| \theta_{\rm new} \|_{H^{1+\alpha }}$ remains under control. Moreover, under such an assumption we can relax the assumptions on $\theta_{\lambda}$. For example, we no longer need it to be the norm inflation solution given by Theorem~\ref{T00}. In fact, we only need it to satisfy a transport equation with appropriate support control, since we are concerned with the control of the outer solution. The only requirement for $\theta_{\lambda}$ comes through its initial condition, where we would require the support included in $B_{2\lambda^{-1/2}}$ and the $L^\infty$ norm bounded  by $\lambda^{-1}$. We will now refer to such inner initial condition as $\theta_\lambda$. 

\begin{lemma}[Uniqueness lemma]\label{uniqueness}
 Let $P\geq 3$, and let $\theta$ be an odd-odd symmetric, $P$-fold symmetric, compactly supported and smooth solution to the SQG equation \eqref{sqg} on $[0,T_0]$ with $\supp \,\theta \subset B_r^c$, and let $\theta_{\rm new}$ be a solution to the SQG equations \eqref{sqg} (in the sense of Definition~\ref{definition_classical_sol}) satisfying 
 \eqnb\label{uni_ass} \| \theta_{\rm new} \|_{H^{1+\alpha}}\leq M
 \eqne
  for $t\in [0,T_0]$, for some $\alpha \in (0,1)$, $M>0$,   with initial condition
\[
\theta (\cdot ,0)+ \theta_\lambda ,
\] 
 where $\theta_\lambda$ is \emph{any initial condition} such that 
 \begin{enumerate}
 \item[(i)] $\supp\, \theta_\lambda \subset B_{2\lambda^{-1/2}}$, 
 \item[(ii)] $\| \theta_\lambda \|_\infty \leq 2\lambda^{-1}$.
\end{enumerate}
Then for every $\epsilon >0$ there exists $\lambda >0$ (large, depending on $T_0,\theta ,M,\epsilon,\alpha$) such that
     $$\|\theta_{\rm new}-\theta \|_{H^{3}(B_{{r}/{2}}^c)}\leq \epsilon $$
     for $t\in [0,T_0]$.
\end{lemma}
We will use the above lemma for the uniqueness claim of Theorem~\ref{T01}, where we will compare two solutions with initial data satisfying \eqref{uni_ass} and with initial condition of the form of infinite union of disjoint annuli concentrating at the origin (see~\eqref{uniqueness_appl}). 
In particular, we note that the lemma allows the support of $\theta_\lambda$ to be any such union, as long as it is a subset of $B_{2\lambda^{-1/2}}$. 
In order to prove Lemma~\ref{uniqueness} we make rigorous the discussion above the lemma.
\begin{proof}
 Note that, by the assumed regularity $\theta_{\rm new}\in H^{1+\alpha}$, we can write
\[
\theta_{\rm new} = \widetilde{\theta}_\lambda + \widetilde{\theta}
\]
with $\widetilde{\theta}_\lambda, \widetilde{\theta}$ (and in particular their supports) transported by $v[\widetilde{\theta}_\lambda + \widetilde{\theta}]$ (see \cite{CPHD}). Thus let us denote by 
\[
\begin{split}
\widetilde{\phi}(x,t) & \text{ the Lagrangian trajectory of } v[\theta_{\rm new} ],\\
{\phi}(x,t) & \text{ the Lagrangian trajectory of } v[\theta ],\end{split}
\]
and let $T\leq T_0$ be the largest time such that
\eqnb\label{suppsa}
|\phi_\lambda (x,t)|\leq \lambda^{-1/4} \text{ for } |x|\leq 2\lambda^{-1/2}\qquad \text{ and } \qquad |\widetilde{\phi} (x,t)|\geq {r/2} \text{ for } x\in \supp \theta(\cdot , 0) 
\eqne
for $t\in [0,T]$. Note that $T>0$ by continuity. Moreover, the transport structure of the SQG equation implies that 
\eqnb\label{supps}
\supp\, \widetilde{\theta}_\lambda \subset \overline{B_{\lambda^{-1/4}}}, \qquad \supp\, \widetilde{\theta} \subset B_{r/2}^c,
\eqne
for $t\in [0,T]$. In particular, similarly as in \eqref{temp_003}, we obtain
\eqnb\label{supps1}
\| v[\widetilde{\theta}_\lambda] \|_{C^4 (B_{r/2}^c)} \lec_r \| \widetilde{\theta}_\lambda \|_{L^1} \lec \lambda^{-2},
\eqne
and so, similarly to \eqref{theta001}, we obtain that
\eqnb\label{supps2}
\| \theta_{\rm new} - \theta \|_{H^3 (B_{r/2}^c)} =\| \widetilde{\theta} - \theta \|_{H^3 (B_{r/2}^c)} \lec_{r,\theta} \lambda^{-2} 
\eqne
for $t\in [0,T]$. We thus need to show that $T= T_0$. If $T<T_0$ then 
\[
\begin{split}
\p_t | \phi - \widetilde{\phi }| &\leq \left| v[\theta ] (\phi ) -v[\theta ] (\widetilde{\phi } ) \right| + \left| v [ \theta_{\rm new} - \theta ] (\widetilde{\phi }) \right| \\
&\leq \| v[\theta ] \|_{C^1} |\phi - \widetilde{\phi}| + C \| \widetilde{\theta} - \theta \|_{H^3(B_{r/2}^c)}+\|v[\widetilde{\theta}_{\lambda}]\| _{L^{\infty}(B_{r/2}^c)}\\
&\lec_{r,\theta,T_0} | \phi - \widetilde{\phi }|  + \lambda^{-2}
\end{split}
\]
for  $x\in \supp \theta (\cdot , 0)$, $t\in [0,T]$.
Thus,  the ODE fact \eqref{ode_fact} gives hat
\[
| \phi - \widetilde{\phi }| \leq C_{r,\theta, T_0} \lambda^{-2},
\]
so that 
\eqnb\label{tocontr}
| \widetilde{\phi }| \geq r -  C_{r,\theta, T_0} \lambda^{-2} \geq 3r/4
\eqne for $x\in \supp \theta (\cdot , 0)$, $t\in [0,T]$, if $\lambda$ is sufficiently large. 

As for $|x| \leq 2\lambda^{-1/2}$ note that 
\eqnb\label{repl}
\|v[\widetilde{\theta}_{\lambda } ]\|_{L^{\infty}}\lec_\alpha \|\widetilde{\theta}_{\lambda }\|_{L^{\infty}}\log(2+\|\widetilde{\theta}_{\lambda }\|_{H^{1+\alpha}}),
\eqne
which we can use to estimate $v[\widetilde{\theta_\lambda}](\widetilde{\phi})$. Thus, we have 
\[
\begin{split}
|\p_t \widetilde{\phi }(x,t) | &= \left| v[\theta_{\rm new} ] (\widetilde{\phi } ) \right| \leq \left| v[\widetilde{\theta}  ] (\widetilde{\phi } ) \right| +\left| v[\widetilde{\theta_{\lambda }}] (\widetilde{\phi } ) \right|  \\
&\leq  \| v[\theta ] \|_{C^1} |\widetilde{\phi } | +\| v[\widetilde{\theta}-\theta ] \|_{L^{\infty}}+C_{\alpha}\| \widetilde{\theta}_{\lambda } \|_{L^\infty} \left( 1+ \log \|\widetilde{\theta}_{\lambda }\|_{H^{1+\alpha }} \right)  \\
&\leq C_{\theta } (|\widetilde{\phi } |+\lambda^{-2}) + \frac{C_{\alpha}}{\lambda } \left( 1+ \log \left( \| \theta_{\rm new } \|_{H^{1+\alpha }}  + \| \widetilde{\theta } - \theta \|_{H^{1+\alpha }} + \|\theta \|_{H^{1+\alpha }} \right) \right) \\
&\leq C_\theta |\widetilde{\phi } | + C_{\alpha,\theta,M }\lambda^{-1}
\end{split}
\]  
for  $|x|\leq 2\lambda^{-1/2}$, $t\in [0,T]$, and so a Gronwall type estimate gives us  
\[
|\widetilde{\phi }(x,t)| \leq |x| + C_{\theta,M,T_{0}} \lambda^{-1}  \leq \lambda^{-1/4}
\]
for such $x,t$, if $\lambda$ is sufficiently large. This and \eqref{tocontr} show that \eqref{suppsa} remains valid for some $t>T$, by continuity. This contradicts the definition of $T$, and thus $T=T_0$.
\end{proof}

 \subsection{Proof of Theorem~\ref{T01}}
 Here we use the existence and uniqueness Lemmas~\ref{existence}--\ref{uniqueness} to conclude the proof of Theorem~\ref{T01}.

Let $T_0$ be given by Theorem~\ref{T00}, and let $P\in \N$ be sufficiently large for Lemma~\ref{existence} to hold. We set 
\eqnb\label{choiceK}
K_i \coloneqq c_0 \frac{P^{\frac{1}{2}}2^{i}}{\varepsilon}, \quad \text{ for }i\geq 1.
\eqne
 For $n\geq 1$ we denote by $\theta^{n}$ the solution to the SQG equation \eqref{sqg} with initial conditions $\sum_{j=1}^{n}\theta_{K_{j},\lambda_j}(x,0)$, where the $\theta_{K_j,\lambda_j}$'s are the solutions to the SQG equation \eqref{sqg} provided by Theorem~\ref{T00}.  Note that, due to the choice \eqref{choiceK} of $K$ we have that $\| \theta_{K_j,\lambda_j}(\cdot ,0)\|_{H^s} \leq \varepsilon 2^{-j}$ for all $j\geq 1$.  We note that $\theta^1$ exists until at least $T_0$ (by Theorem~\ref{T00}) and is supported in $B_{\varepsilon/2}$, provided $\lambda_1$ is chosen large enough. Moreover, each $\theta^n$  is a smooth odd-odd symmetric, $P$-fold symmetric function, exists until at least $T_0$ by Lemma~\ref{existence} and for $r_n=\frac{1}{4(\lambda_n)^{\frac12}}$ we have
\eqnb\label{supp_of_thetan}
 \supp\, \theta^n (\cdot , t) \subset B_{r_n}^c \qquad \text{ for all }t\in [0,T_0],
 \eqne
 by Lemma~\ref{L_approach}, provided $\lambda_2, \lambda_3,\ldots $ are chosen large enough. Apart from this, given $\lambda_1,\ldots, \lambda_n$, we pick $\lambda_{n+1}$ even larger so that
 \begin{enumerate}
       \item[(i)]  $\lambda_i \geq \exp(K_i)$,
       \item[(ii)] $\lambda_{n+1}^{-\frac12} \leq 2^{-(n+2)}$ and $4\lambda_{n+1}^{-\frac12} < \lambda_n^{-\frac12}/4$ (so that, respectively,  the support of the $n+1$-st piece, $\theta_{K_{n+1},\lambda_{n+1}}$, stays inside $B_{2^{-(n+1)}}$ and that all pieces are separated, i.e. $\text{supp }\theta_{K_{n+1},\lambda_{n+1}} \cap  \theta_{K_{j},\lambda_{j}}=\emptyset$ for $j=1,\ldots ,n$) and $\lambda_{n+1}^{-1}\leq \lambda_n^{-1} /4$ (so that the $L^\infty$ norm of the pieces $\geq n+1$ is bounded by $3\lambda_{n+1}^{-1}/2$),
      \item[(iii)] $\|\theta^{n+1}-\theta^{n}-\theta_{K_{n+1},\lambda_{n+1}}\|_{H^3}\leq \varepsilon 2^{-n-1}$ for $t\in [0,T_0 ]$, 
       which is possible by Lemma~\ref{existence} (i.e. take $\lambda_{n+1}\geq 2^{n+1}\varepsilon^{-1}$).
       \item[(iv)] The uniqueness Lemma~\ref{uniqueness} holds with $M=n+1$, $\alpha=\frac{1}{n+1}$, $\epsilon\leq \frac{1}{n}$, $\theta=\theta^{n}$.
       
         \end{enumerate}

We note that  (ii) and (iii)  imply that $\theta^n$ converges, in the sup norm  to some limit $\theta^\infty \in C(\R^2\times [0,T_0])$ with compact support in space. To be precise, (iii) implies that, for $m\geq n$,    
   \begin{equation}\label{convH3cauchy}
       \left\|\theta^{m}-\theta^{n}-\sum_{j=n+1}^{m}\theta_{K_{j},\lambda_{j}}\right\|_{H^3}\leq \varepsilon \sum_{j=n+1}^{m}2^{-j}\leq \varepsilon 2^{-n},
   \end{equation}
uniformly in $t\in [0,T_0 ]$. Thus, for $m\geq n$
\eqnb\label{gluing_Cauchy}
\begin{split}
\| \theta^m - \theta^n \|_{C(\overline{B_1})} &\leq \left\| \theta^m - \theta^n - \sum_{j=n+1}^m \theta_{K_j,\lambda_j} \right\|_{C(\overline{B_1})} +\sum_{j=n+1}^m \| \theta_{K_j,\lambda_j}  \|_{C(\overline{B_1})} \\
&\leq C\varepsilon 2^{-n} + \sum_{j>n} \| \theta_{K_j,\lambda_j}  \|_{L^\infty} \lec \varepsilon 2^{-n} \to 0 
\end{split}
\eqne
as $n\to \infty$, where we used (ii) in the last inequality (recall \eqref{black_box_claims} that $\| \theta_{K_j,\lambda_j} \|_{L^\infty} \leq \lambda_j^{-1}$), and so $\{ \theta^n \}$ is Cauchy in $C(\overline{B_1}\times [0,T_0 ])$. Thus there exists  $\theta^\infty \in C(\overline{B_1}\times [0,T_0])$ such that $\supp\, \theta^\infty \subset B_\varepsilon$ and $\theta^n \to  \theta^\infty$ in $C(\overline{B_1}\times [0,T_0])$. Moreover, taking $m\to \infty$ in \eqref{convH3cauchy} and \eqref{gluing_Cauchy} gives, respectively, 
\eqnb\label{linfty}
\| \theta^n - \theta^\infty \|_{C(\overline{B_1}\times [0,T_0])} \lec \varepsilon 2^{-n}
\eqne
and  
 \eqnb\label{error_h3} 
 \left\|\theta^{\infty}-\theta^n - \sum_{j\geq n+1}\theta_{K_{j},\lambda_{j}}\right\|_{H^3}\leq \varepsilon 2^{-n}
 \eqne
 for all $n\geq 1$. We emphasize that the last inequality estimates the error between $\theta^\infty$ and $\theta^n +\sum_{j\geq n+1}\theta_{K_{j},\lambda_{j}}$, while each of these functions does not belong to $H^3$ (in fact the theorem claims that  $\theta^\infty$ does not belong even to Sobolev spaces below $H^2$). In order to verify \eqref{error_h3} we first note that \eqref{convH3cauchy} gives that
\[
\left\| g_m \right\|_{H^3} \leq \varepsilon 2^{-n}
\]
for all $m\geq 1$, where $g_m\coloneqq \theta^m -\theta^n - \sum_{j=n+1}^m \theta_{K_j,\lambda_j}$, and $n\geq 1$ is fixed. Thus, by reflexivity of $H^3$, there exists a subsequence $g_{m_k}$ that converges weakly in $H^3$ to some function $g\in H^3$ with $\| g \|_{H^3}\leq \varepsilon 2^{-n}$. However, $g_m \to \theta^\infty -\theta^n- \sum_{j\geq n+1} \theta_{K_j,\lambda_j } $ in the supremum norm (in particular the infinite sum is understood in the sense of the limit in the supremum norm). Thus $g= \theta^\infty -\theta^n- \sum_{j\geq n+1} \theta_{K_j,\lambda_j }$ and \eqref{error_h3} follows.
   
 We note that, for each $\delta>0$, functions $\theta_{K_j,\lambda_j}$, $j\geq n$ are supported inside $B_\delta$ if $n$ is taken sufficiently large. This lets us use \eqref{error_h3} to obtain
   \[
   \|\theta^{\infty}-\theta^{n}\|_{H^3(B_{\delta}^c)} \to 0 \hspace{1cm} \text{ as } n\to \infty.
   \]
   The same observation lets us estimate, for $j\geq n+1$, derivatives of $v[\theta_{K_j,\lambda_j}]$ by  $C_\delta \| \theta_{K_j,\lambda_j} \|_{L^\infty}$, so that, for each $\delta >0$,
   \[
   \begin{split}
   \| v [\theta^\infty - \theta^n ] \|_{H^3(B_\delta^c)}&\leq  \left\| v \left[\theta^\infty - \theta^n -\sum_{j\geq n+1} \theta_{K_j,\theta_j} \right] \right\|_{H^3}+ \sum_{j\geq n+1} \left\| v \left[ \theta_{K_j,\theta_j} \right] \right\|_{H^3 (B_\delta^c )} \\
   &\lec_\delta  \varepsilon 2^{-n} +  \sum_{j\geq n+1} \| \theta_{K_j,\lambda_j }\|_{L^\infty } \to 0 
   \end{split}
 \]
   as $n\to \infty$. In particular we obtain that, for each $\delta>0$,
   \[ \| \theta^{\infty}-\theta^{n}\|_{C^{2-\delta}(B_{\delta}^c )},\| v[\theta^{\infty}-\theta^{n}]\|_{C^{2-\delta}(B_{\delta}^c )} \to 0 
   \]
   as $n\to \infty$.   Thus, we can integrate the SQG equation \eqref{sqg}, for each $\theta^n$, in time and take $n\to \infty$ to obtain 
   $$\theta^{\infty}(x,t_{2})-\theta^{\infty}(x,t_{1})=-\int_{t_{1}}^{t_{2}}v[\theta^{\infty}](x,s)\cdot\nabla\theta^{\infty}(x,s)\,\d s $$
   for every $x\in \R^2 \setminus \{ 0 \}$, $t_1,t_2\in [0,T_0]$. Dividing by $t_2-t_1$ and taking $t_2\to t_1$ we obtain that $\p_t \theta^\infty$ exists on $\R^2\setminus \{ 0 \} \times [0,T_0]$ and that the SQG equation \eqref{sqg} holds for $\theta^\infty$ everywhere, except for $x=0$. 
   
   As for $x=0$, we notice that, due to the odd symmetry of $\theta^{\infty}$ we have that 
   $$\theta^{\infty}(0,t)=0,\quad v[\theta^{\infty}](0,t)=0.$$
In order to verify that spatial derivatives of $\theta^\infty$ exist at $x=0$ we note that
\[
\frac{\theta^\infty (h \, e_j)}{h}= \frac{\theta^\infty (h \, e_j)-\sum_{j\geq 1} \theta_{K_j,\lambda_j}(h \, e_j)}{h} + \sum_{j\geq 1} \frac{\theta_{K_j,\lambda_j}(h \, e_j)}{h}
\]   
   for $j=1,2$, $h>0$. The first term on the right-hand side converges as $h\to 0$ since $\theta^\infty - \sum_{j\geq 1} \theta_{K_j,\lambda_j}$ is a $C^1$ function (by \eqref{error_h3}, taken with $n=1$), and the second term converges to $0$, since each ingredient vanishes if $h\, e_j \not \in \supp\, \theta_{K_j,\lambda_j}$ and
   \[
\left|    \sum_{j\geq 1} \frac{\theta_{K_j,\lambda_j}(h \, e_j)}{h} \right|\leq \sum_{j\geq 1} \frac{\chi_{\{ h \sim \lambda_j^{-1/2} \} } \| \theta_{K_j,\lambda_j} \|_{C^0} }{h} \lec \sum_{j\geq 1}\chi_{\{ h \sim \lambda_j^{-1/2} \} }\lambda_j^{-1/2} \to 0
   \]
   as $h\to 0$,    where we used the fact that $\| \theta_{K_j,\lambda_j} \|_{L^\infty } \leq \lambda_j^{-1}$ (recall~\eqref{black_box_claims}) in the last inequality, and $\chi_E$ denotes the indicator function of a set $E$. Thus $\nabla \theta^\infty (0,t)$ exists for all $t\in [0,T]$. Consequently, and since $\theta^{\infty}(x=0,t)=0$, the SQG equation holds trivially at $x=0$, and hence $\theta^\infty$ is a classical solution in the sense of Definition~\ref{definition_classical_sol}, as needed.
   
Clearly, the initial condition for $\theta^\infty $ is
   $$\theta^\infty (x,0) = \sum_{j=1}^{\infty}\theta_{K_{j},\lambda_{j}}(x,0),$$
   and so our choice \eqref{choiceK} of the $K_i$'s implies that 
  \eqnb\label{choiceKcons}
  \|\theta^{\infty}(x,0)\|_{H^{s}}\leq \varepsilon,
  \eqne
   as required. 
   
   In order to check the norm growth, we use \eqref{error_h3} and apply Lemma~\ref{L_ss_cons} (with $\delta \coloneqq 1/2$) to obtain
   \[
   \begin{split}
   \|\theta^{\infty}\|_{H^{\beta}}&= \left\|\sum_{j\geq 1}\theta_{K_{j},\lambda_{j}}\right\|_{H^{\beta}}+ O( \varepsilon )\\
   &= \left( \sum_{j\geq 1}\left\|\theta_{K_{j},\lambda_{j}}\right\|_{H^{\beta}}^2  + O\left(  \sum_{j\geq 1} \lambda_j^{\beta -3/4} \left\|\nabla \theta_{K_{j},\lambda_{j}}\right\|_{L^2}^2 \right) \right)^{\frac{1}{2}}+ O( \varepsilon )\\
   &= \left( \sum_{j\geq 1}\left\|\theta_{K_{j},\lambda_{j}}\right\|_{H^{\beta}}^2 + O(1)  \right)^{\frac{1}{2}}+O(\varepsilon ),
   \end{split}
   \]
  
   where we used the fact that $\| \theta_{K_j,\lambda_j} \|_{H^1} \lec_{P}  \lambda_j^{-1} \log \lambda_j$ (recall \eqref{black_box_claims}) to obtain
\[
\sum_{j\geq 1} \lambda_j^{\beta -3/4} \left\| \nabla \theta_{K_{j},\lambda_{j}}\right\|_{L^2}^2 \lec_P  \sum_{j\geq 1} K_j^{-2} \lambda_j^{\beta -11/4} \log \lambda_j  \lec 1 .
\]   
We now consider $\beta <s$, and note that \eqref{black_box_Hb} and the above equality imply that
\eqnb\label{ooo}
\| \theta^{\infty } (\cdot , t ) \|_{H^\beta } < \infty \quad \Leftrightarrow \quad \sum_{j\geq 1} K_{j}^{-1-\frac{\beta t}{s-1}} \lambda_{j}^{c_{2}(\beta-s)+c_{3}\beta t} (\log \lambda_j )^{-c_4 \beta t} <\infty .
\eqne
   Thus, $ \| \theta^{\infty } (\cdot , t ) \|_{H^\beta } < \infty$ if $c_{2}(\beta-s)+c_{3}\beta t\leq 0$ (as $\sum_{j\geq 1} K_j^{-1} <\infty$). Otherwise, if $c_{2}(\beta-s)+c_{3}\beta t> 0$, then $ \| \theta^{\infty } (\cdot , t ) \|_{H^\beta } = \infty$,  since the $\lambda_j$'s dominate the $K_j$'s, due to (i). This gives the loss of regularity claim in Theorem~\ref{T01}, since
   \[
   c_{2}(\beta-s)+c_{3}\beta t> 0 \quad \Leftrightarrow \quad \beta > \frac{s}{1+\overline{c}t},
   \]
   where $\overline{c} \coloneqq c_3/c_2$.\\
   
   As for the uniqueness claim, suppose that there exists another solution $\widetilde{\theta}^{\infty}$ with initial conditions $\theta^{\infty}(x,0)$ and such that, for some $T\in (0,T_0]$,   
   $$\|\widetilde{\theta}^{\infty}\|_{H^{1+\alpha}}\leq M$$
for $t\in[0,T]$. By (iv) we can apply the uniqueness Lemma~\ref{uniqueness} for $n>\text{max}(M,\frac{1}{\alpha})$. Indeed, since the initial condition for both $\theta^\infty$ and $\widetilde{\theta}$ is $\theta^n (\cdot ,0) + \theta_{\lambda_{n+1}}$, where
\[\theta_{\lambda_{n+1}} \coloneqq  \sum_{j\geq n+ 1} \theta_{K_j,\lambda_j}(\cdot ,0)\]
satisfies $\supp \, \theta_{\lambda_{n+1}} \subset B_{2\lambda_{n+1}^{-\frac{1}{2}}}$ and $\| \theta_{\lambda_{n+1}} \|_\infty \leq 2 \lambda_{n+1}^{-1}$ (by (ii)), the claim of Lemma~\ref{uniqueness} applies to both $\theta^\infty - \theta^n$ and $\widetilde{\theta}^\infty - \theta^n$ to give
\eqnb\label{uniqueness_appl}
   \|\theta^{\infty}-\widetilde{\theta}^{\infty}\|_{H^{3}\left(  B_{r_n/2}^c \right)}\leq \|\theta^{\infty}-\theta^{n}\|_{H^{3}\left(  B_{r_n/2}^c \right)}+\|\theta^{n}-\widetilde{\theta}^{\infty}\|_{H^{3}\left(  B_{r_n/2}^c \right)}\leq \frac{2}{n}.
   \eqne
   Since $r_n\to 0$ we thus obtain that 
   \[
    \|\theta^{\infty}-\widetilde{\theta}^{\infty}\|_{H^{3}\left(  B_{\eta}^c \right)} = 0
   \]
   for every $\eta >0$, $t\in [0,T_0]$, and so 
    $\theta^\infty = \widetilde{\theta}^\infty $ on $(\R^2 \setminus \{ 0 \} ) \times [0,T_0]$, and continuity of both ${\theta}^\infty$ and $\widetilde{\theta}^\infty$ (recall Definition~\ref{definition_classical_sol}) gives $ \theta^\infty = \widetilde{\theta}^\infty $, as required.   

\section*{Acknowledgements}
This work was partially supported by the  NSF grant no.~DMS-2511556 and by the Thematic Research Programme, University of Warsaw, Excellence Initiative Research University. This work is also supported in part by the Spanish Ministry of Science
and Innovation, through the “Severo Ochoa Programme for Centres of Excellence in R$\&$D (CEX2019-000904-S \& CEX2023-001347-S)” and 152878NB-I00. We were also partially supported by the ERC Advanced Grant 788250, and by the SNF grant FLUTURA: Fluids, Turbulence, Advection No. 212573.

\bibliographystyle{plain}
\bibliography{literature}

\newcommand{\etalchar}[1]{$^{#1}$}
\begin{thebibliography}{ADdP{\etalchar{+}}21}

\bibitem[ADdP{\etalchar{+}}21]{ADMW}
W.~Ao, J.~D\'avila, M.~del Pino, M.~Musso, and J.~Wei.
\newblock Travelling and rotating solutions to the generalized inviscid surface
  quasi-geostrophic equation.
\newblock {\em Trans. Amer. Math. Soc.}, 374(9):6665--6689, 2021.

\bibitem[BC94]{BC}
H.~Bahouri and J.-Y. Chemin.
\newblock \'equations de transport relatives \'a\ des champs de vecteurs
  non-lipschitziens et m\'ecanique des fluides.
\newblock {\em Arch. Rational Mech. Anal.}, 127(2):159--181, 1994.

\bibitem[BCMZ24]{BCM}
R.~Bianchini, D.~C\'ordoba, and L.~Mart\'inez-Zoroa.
\newblock Non existence and strong ill-posedness in $h^2$ for the stable {IPM}
  {E}quation.
\newblock 2024.
\newblock arXiv:2410.01297.

\bibitem[BHP23]{BHP}
A.~Bulut, M.K.~Huynh Huynh, and S.~Palasek.
\newblock Non-uniqueness up to the onsager threshold for the forced sqg
  equation.
\newblock 2023.
\newblock arXiv:2310.12947.

\bibitem[BL15a]{BL1}
J.~Bourgain and D.~Li.
\newblock Strong ill-posedness of the incompressible {E}uler equation in
  borderline {S}obolev spaces.
\newblock {\em Invent. Math.}, 201(1):97--157, 2015.

\bibitem[BL15b]{BL2}
J.~Bourgain and D.~Li.
\newblock Strong illposedness of the incompressible {E}uler equation in integer
  {$C^m$} spaces.
\newblock {\em Geom. Funct. Anal.}, 25(1):1--86, 2015.

\bibitem[BSV19]{BSV}
T.~Buckmaster, S.~Shkoller, and V.~Vicol.
\newblock Nonuniqueness of weak solutions to the {SQG} equation.
\newblock {\em Comm. Pure Appl. Math.}, 72(9):1809--1874, 2019.

\bibitem[C\'98]{cordoba}
D.~C\'ordoba.
\newblock Nonexistence of simple hyperbolic blow-up for the quasi-geostrophic
  equation.
\newblock {\em Ann. of Math. (2)}, 148(3):1135--1152, 1998.

\bibitem[CCGS20]{CCGS}
A.~Castro, D.~C\'ordoba, and J.~G\'omez-Serrano.
\newblock Global smooth solutions for the inviscid {SQG} equation.
\newblock {\em Mem. Amer. Math. Soc.}, 266(1292):v+89, 2020.

\bibitem[CF02]{CF}
D.~C\'ordoba and C.~Fefferman.
\newblock Growth of solutions for {QG} and 2{D} {E}uler equations.
\newblock {\em J. Amer. Math. Soc.}, 15(3):665--670, 2002.

\bibitem[CKL21]{ChKL}
X.~Cheng, H.~Kwon, and D.~Li.
\newblock Non-uniqueness of steady-state weak solutions to the surface
  quasi-geostrophic equations.
\newblock {\em Comm. Math. Phys.}, 388(3):1281--1295, 2021.

\bibitem[CLS{\etalchar{+}}12]{CLSTW}
P.~Constantin, M.-C. Lai, R.~Sharma, Y.-H. Tseng, and J.~Wu.
\newblock New numerical results for the surface quasi-geostrophic equation.
\newblock {\em J. Sci. Comput.}, 50(1):1--28, 2012.

\bibitem[CMT94]{CMT}
P.~Constantin, A.~J. Majda, and E.~Tabak.
\newblock Formation of strong fronts in the {$2$}-{D} quasigeostrophic thermal
  active scalar.
\newblock {\em Nonlinearity}, 7(6):1495--1533, 1994.

\bibitem[CMZ22]{CM}
D.~C\'ordoba and L.~Mart\'inez-Zoroa.
\newblock Non existence and strong ill-posedness in {$C^k$} and {S}obolev
  spaces for {SQG}.
\newblock {\em Adv. Math.}, 407:Paper No. 108570, 74, 2022.

\bibitem[CMZ24a]{CM2}
D.~C\'ordoba and L.~Mart\'inez-Zoroa.
\newblock Global unique solutions with instantaneous loss of regularity for
  {SQG} with fractional diffusion.
\newblock {\em Ann. PDE}, 10(2):Paper No. 21, 52, 2024.

\bibitem[CMZ24b]{CM1}
D.~C\'ordoba and L.~Mart\'inez-Zoroa.
\newblock Non-existence and strong ill-posedness in {$C^{k,\beta}$} for the
  generalized surface quasi-geostrophic equation.
\newblock {\em Comm. Math. Phys.}, 405(7):Paper No. 170, 53, 2024.

\bibitem[CMZO24]{CMO}
D.~C\'ordoba, L.~Mart\'inez-Zoroa, and W.~S. O\.za\'nski.
\newblock Instantaneous gap loss of {S}obolev regularity for the 2{D}
  incompressible {E}uler equations.
\newblock {\em Duke Math. J.}, 173(10):1931--1971, 2024.

\bibitem[CN18]{CN}
P.~Constantin and H.~Q. Nguyen.
\newblock Local and global strong solutions for {SQG} in bounded domains.
\newblock {\em Phys. D}, 376/377:195--203, 2018.

\bibitem[Cri08]{CPHD}
G.~Crippa.
\newblock {\em The flow associated to weakly differentiable vector fields}.
\newblock PhD thesis, University of Zurich, 2008.

\bibitem[CW12]{chae_wu}
D.~Chae and J.~Wu.
\newblock Logarithmically regularized inviscid models in borderline {S}obolev
  spaces.
\newblock {\em J. Math. Phys.}, 53(11):115601, 15, 2012.

\bibitem[DGR24]{DGR}
M.~Dai, V.~Giri, and R.-O. Radu.
\newblock An onsager-type theorem for sqg.
\newblock 2024.
\newblock arXiv:2407.02582.

\bibitem[DNPV12]{hitchhiker}
E.~Di~Nezza, G.~Palatucci, and E.~Valdinoci.
\newblock Hitchhiker's guide to the fractional {S}obolev spaces.
\newblock {\em Bull. Sci. Math.}, 136(5):521--573, 2012.

\bibitem[DP23a]{DP1}
M.~Dai and Q.~Peng.
\newblock Non-unique stationary solutions of forced sqg.
\newblock 2023.
\newblock arXiv:2302.03283.

\bibitem[DP23b]{DP2}
M.~Dai and Q.~Peng.
\newblock Non-unique weak solutions of forced sqg.
\newblock 2023.
\newblock arXiv:2310.13537.

\bibitem[EJ17]{EJ}
T.~M. Elgindi and I.-J. Jeong.
\newblock Ill-posedness for the incompressible {E}uler equations in critical
  {S}obolev spaces.
\newblock {\em Ann. PDE}, 3(1):Paper No. 7, 19, 2017.

\bibitem[EM20]{EM}
T.~M. Elgindi and N.~Masmoudi.
\newblock {$L^\infty$} ill-posedness for a class of equations arising in
  hydrodynamics.
\newblock {\em Arch. Ration. Mech. Anal.}, 235(3):1979--2025, 2020.

\bibitem[FS05]{FS}
S.~Friedlander and R.~Shvydkoy.
\newblock The unstable spectrum of the surface quasi-geostropic equation.
\newblock {\em J. Math. Fluid Mech.}, 7:S81--S93, 2005.

\bibitem[GS19]{GS}
P.~Gravejat and D.~Smets.
\newblock Smooth travelling-wave solutions to the inviscid surface
  quasi-geostrophic equation.
\newblock {\em Int. Math. Res. Not. IMRN}, (6):1744--1757, 2019.

\bibitem[HK21]{HK}
S.~He and A.~Kiselev.
\newblock Small-scale creation for solutions of the {SQG} equation.
\newblock {\em Duke Math. J.}, 170(5):1027--1041, 2021.

\bibitem[HPGS95]{HPGS}
I.~M. Held, R.~T. Pierrehumbert, S.~T. Garner, and K.~L. Swanson.
\newblock Surface quasi-geostrophic dynamics.
\newblock {\em J. Fluid Mech.}, 282:1--20, 1995.

\bibitem[IM21]{IM1}
P.~Isett and A.~Ma.
\newblock A direct approach to nonuniqueness and failure of compactness for the
  {SQG} equation.
\newblock {\em Nonlinearity}, 34(5):3122--3162, 2021.

\bibitem[IM24]{IM2}
P.~Isett and A.~Ma.
\newblock On the conservation laws and the structure of the nonlinearity for
  sqg and its generalizations.
\newblock 2024.
\newblock arXiv:2403.08279.

\bibitem[Jeo21]{JEuler}
I.-J. Jeong.
\newblock Loss of regularity for the 2{D} {E}uler equations.
\newblock {\em Journal of Mathematical Fluid Mechanics}, 23:1--11, 2021.

\bibitem[JK24]{JK}
I.-J. Jeong and J.~Kim.
\newblock Strong ill-posedness for {SQG} in critical {S}obolev spaces.
\newblock {\em Anal. PDE}, 17(1):133--170, 2024.

\bibitem[JKM22]{JKM}
M.~S. Jolly, A.~Kumar, and V.~R. Martinez.
\newblock On local well-posedness of logarithmic inviscid regularizations of
  generalized {SQG} equations in borderline {S}obolev spaces.
\newblock {\em Commun. Pure Appl. Anal.}, 21(1):101--120, 2022.

\bibitem[JMZO25]{JMO}
I.-J. Jeong, L.~Mart\'inez-Zoroa, and W.~S. O\.za\'nski.
\newblock Instantaneous continuous loss of sobolev regularity for the 3{D}
  incompressible {E}uler equations.
\newblock 2025.
\newblock arXiv:2508.06333.

\bibitem[KN12]{KN}
A.~Kiselev and F.~Nazarov.
\newblock A simple energy pump for the surface quasi-geostrophic equation.
\newblock In {\em Nonlinear partial differential equations}, volume~7 of {\em
  Abel Symp.}, pages 175--179. Springer, Heidelberg, 2012.

\bibitem[Mar08]{marchand}
F.~Marchand.
\newblock Existence and regularity of weak solutions to the quasi-geostrophic
  equations in the spaces {$L^p$} or {$\dot H^{-1/2}$}.
\newblock {\em Comm. Math. Phys.}, 277(1):45--67, 2008.

\bibitem[OY97]{OY}
K.~Ohkitani and M.~Yamada.
\newblock Inviscid and inviscid-limit behavior of a surface quasigeostrophic
  flow.
\newblock {\em Phys. Fluids}, 9(4):876--882, 1997.

\bibitem[Ped79]{pedlosky}
J.~Pedlosky.
\newblock {\em Geophysical fluid dynamics}.
\newblock Springer, 1979.

\bibitem[Res95]{resnick}
S.~G. Resnick.
\newblock Dynamical problems in non-linear advective partial differential
  equations.
\newblock 1995.
\newblock PhD thesis, University of Chicago, Department of Mathematics.

\bibitem[Sco11]{scott}
R.~K. Scott.
\newblock A scenario for finite-time singularity in the quasigeostrophic model.
\newblock {\em J. Fluid Mech.}, 687:492--502, 2011.

\bibitem[Wu05]{Wu}
J.~Wu.
\newblock Solutions of the 2{D} quasi-geostrophic equation in {H}\"older
  spaces.
\newblock {\em Nonlinear Anal.}, 62(4):579--594, 2005.

\end{thebibliography}

\end{document}